\RequirePackage{fix-cm}
\documentclass[smallextended]{svjour3}       
\smartqed  
\usepackage{graphicx}
\usepackage{graphicx, amssymb, amsmath}

\usepackage{indentfirst}
\usepackage{mathrsfs}
\usepackage{amsfonts}
\usepackage{color}
\usepackage{algorithm}
\usepackage{algorithmic}
\usepackage{upgreek}
\usepackage{subfigure}
\usepackage[misc]{ifsym}
\usepackage[colorlinks,linkcolor=blue,anchorcolor=blue,citecolor=blue]{hyperref}  

\usepackage{booktabs}    
\usepackage{multirow}
\usepackage{makecell}

\usepackage{epstopdf}   

\usepackage [latin1]{inputenc}

\usepackage{geometry}
\newcommand{\be}{\begin{equation}}
\newcommand{\ee}{\end{equation}}
\newcommand{\ba}{\begin{array}}
\newcommand{\ea}{\end{array}}
\newcommand{\beas}{\begin{eqnarray*}}
\newcommand{\eeas}{\end{eqnarray*}}
\newcommand{\bea}{\begin{eqnarray}}
\newcommand{\eea}{\end{eqnarray}}
\newcommand{\ome}{\Omega}

\newcommand{\blem}{\begin{lemma}}
\newcommand{\elem}{\end{lemma}}
\newcommand{\bthe}{\begin{theorem}}
\newcommand{\ethe}{\end{theorem}}

\newcommand{\bprop}{\begin{proposition}}
\newcommand{\eprop}{\end{proposition}}

\begin{document}

\title{Local Averaging Type a Posteriori Error Estimates for the Nonlinear
Steady-state Poisson-Nernst-Planck Equations}




\author{Ying Yang$^1$
 \and Ruigang Shen$^2$
 \and Mingjuan Fang$^3$
 \and Shi Shu$^{4,*}$
}


\institute{ \Letter \quad Shi Shu \\
             {\color{white}{ }}\quad~ shushi@xtu.edu.cn \\
             \\
             {\color{white}{ }}\quad~ Ying Yang \\
             {\color{white}{ }}\quad~ yangying@lsec.cc.ac.cn \\
             \\
             {\color{white}{ }}\quad~ Ruigang Shen \\
             {\color{white}{ }}\quad~ shenruigang@163.com    \\
              \\
             {\color{white}{ }}\quad~ Mingjuan Fang \\
             {\color{white}{ }}\quad~ mingjuanfang2012@163.com  \\
             \\
         {\color{white}}$^1$ \quad School of Mathematics and Computational Science, Guangxi Colleges and Universities Key Laboratory of Data
         {\color{white}} \quad~ Analysis and Computation, Guilin University of Electronic Technology, Guilin 541004, People's Republic of China
         \\
          \\
         {\color{white}}$^2$ \quad School of  Mathematics and Computational Science, Hunan Key Laboratory for Computation and Simulation in {\color{white}} \quad~ Science and Engineering, Xiangtan University, Xiangtan  411105, People's Republic of China
          \\
          \\
         {\color{white}}$^3$ \quad School of Mathematics and Computational Science, Guilin University of Electronic Technology, Guilin 541004, \\ {\color{white}} \quad~ Guangxi, People's Republic of  China
         \\
         \\
         {\color{white}}$^{4,*}$  \textbf{Corresponding author.} Hunan Key Laboratory for Computation and Simulation in Science and Engineer-\\
         {\color{white}} \quad~ ing, Key Laboratory of Intelligent Computing and Information Processing of Ministry of Education, Xiangtan \\
          {\color{white}} \quad~ University, Xiangtan 411105, People's Republic of China
}

\date{Received: date / Accepted: date}

\maketitle

\begin{abstract}
 The a posteriori error estimates are studied for a class of nonlinear stead-state Poisson-Nernst-Planck equations, which are a coupled system consisting of the Nernst-Planck equation and the Poisson equation. Both the global upper bounds and the local lower bounds of the error estimators are obtained by using a local averaging operator. Numerical experiments are given to confirm the reliability and efficiency of the error estimators.

\keywords{Poisson-Nernst-Planck equations \and a posteriori error estimates \and finite element \and local averaging \and nonlinear}


\subclass{65N15 \and 65N30}

\end{abstract}

\section{Introduction}{\label{sec1}
 \noindent
 In this paper, we study the a posteriori error estimates for a class of nonlinear steady-state Poisson-Nernst-Planck (PNP) equations.
 The classic PNP equations were originally proposed by Nernst \cite{W.Nernst1889} and Planck \cite{M.Planck1890} which were used to describe the mass conservation of ions and the process of the electrostatic diffusion-reaction.
 As an important mathematical model to describe the ion transport, PNP equations have been widely applied to study the transport of charged particles in semiconductors \cite{J.Jer1996,S.Sel1984}, electrochemical systems \cite{M.Bazant2009,G.Rich2007,M.van2010}, the process of the electrostatic diffusion-reaction \cite{B.Z.Lu2011,bzl2007}, and ion conversion between biological membrane channels \cite{rdcoa2005,asing2009}, etc.

 However, the classical PNP equations have some drawbacks in simulating the physical or biological phenomenon in some practical problems. For example, the PNP model cannot reflect the effects caused by the ion size effect when it is used to simulate the experimental data of the ion channel. However, these effects are of great importance in determining selectivity of channels and the properties of ionic solutions in general \cite{TLHorng2012}. In order to observe and study the biochemical phenomena in the experiment more precisely and then analyze the corresponding diffusion phenomena and principles in detail, some modified PNP equations are presented to deal with the existing limitations. Lu and Zhou \cite{B.Z.Lu2011} proposed a class of nonlinear PNP equations including the ion size effect. Compared with the classic PNP equations, the nonlinear PNP equations are more effective in simulating the biomolecular diffusion-reaction processes. By taking the protein (ion channel) structure into account, Hyon et al. \cite{Y.K.Hyon2014} developed a class of nonlinear PNP system for ion channel. Compared with the primitive PNP model, these modifications in PNP models cause strong nonlinearity, which brings many difficulties in analysis and computation.

 Due to the coupling between the electrostatic potential and concentrations of the ionic species, the PNP system can hardly be solved analytically. Hence, there appears a lot of literature on numerical methods for PNP equations, including the finite difference method \cite{A.Flavell2014,D.He2016,H.Liu2014}, the finite volume method \cite{Cha-peng2004,Bess2014}, and the finite element method \cite{B.Z.Lu2011,Lubz2010,H.D.Gao2017}, etc. In terms of error analysis, there are some work on the finite element method. In \cite{yanglu2013}, Yang and Lu presented a finite element error analysis for a type of steady-state PNP equations modeling the electrodiffusion of ions in a solvated biomolecular system. Sun et al. \cite{Y.Z.Sun2016} analyzed the a priori error estimates of the finite element approximation to a type of time-dependent PNP equations, in which a fully implicit nonlinear Crank-Nicolson scheme is studied and the optimal $H^1$ norm error estimate is obtained for both the ion concentration and electrostatic potential. Gao and He \cite{H.D.Gao2017} constructed a linearized conservative finite element method to discrete the PNP system with zero Neumann boundary conditions and established unconditionally optimal error estimates in $L^2$ norm. The superconvergence analysis of finite element method for the time-dependent PNP equations is studied by Shi and Yang in \cite{Shidy2019}. Besides, in order to obtain the optimal error estimates in $L^2$ norm for both the electrostatic potential and the ionic concentrations, a mixed finite element method is also studied for PNP equations, see \cite{HeSun2017,HeSun2018} for more details. Recently, by introducing a similar projection operator as in \cite{H.D.Gao2017}, Shen et al. \cite{shen2019} presented the optimal error estimate in $L^2$ norm for both the semi- and full implicit nonlinear schemes for the time-dependent PNP equations.

 Although there has been some work on the a priori error analysis of the finite element method for PNP equations, to the best of authors' knowledge, there is no work on the a posteriori error analysis for PNP equations.
 The main purpose of this paper is to provide a complete a posteriori error analysis for the finite element approximation to a class of nonlinear steady-state PNP equations. We consider the following generic nonlinear PNP problem
 \begin{align} \label{MPNP-model}
  \left\{\begin{array}{l}
  \mathcal{L}(p^i,\phi) = -\nabla\cdot\big( \alpha(x,p^i)\nabla p^i
   + \beta(x,p^i) + \gamma(x,p^i)\nabla\phi \big)+g(x,p^i)=0, ~~\mbox{in}~~ \Omega,~1\leq i\leq n,\\
  -\nabla\cdot(\epsilon(x)\nabla\phi) - \displaystyle{\sum_{i=1}^n} q^ip^i=f, ~~\mbox{in}~~ \Omega,
  \end{array}\right.
 \end{align}
 with homogeneous Dirichlet boundary conditions
 \begin{align} \label{MPNP-model-bd0}
  p^i=\phi = 0,~~\mbox{on}~~ \partial\Omega,
 \end{align}
 where $\Omega \subset \mathbb{R}^2$ is a polygonal domain with a Lipschitz boundary $\partial\Omega$, $p^i$ is the concentration of the $i$-th ionic species with charge $q^i$ and $\phi$ is the electrostatic potential. The coefficients $\alpha(x,y):\bar{\Omega}\times \mathbb{R}^1\rightarrow \mathbb{R}^1$, $\beta(x,y):\bar{\Omega}\times \mathbb{R}^1\rightarrow \mathbb{R}^2$,  $\gamma(x,y):\bar{\Omega}\times \mathbb{R}^1\rightarrow \mathbb{R}^1$, $g(x,y):\bar{\Omega}\times \mathbb{R}^1\rightarrow \mathbb{R}^1$ and the dielectric coefficient $\epsilon(x):\bar{\Omega}\times \mathbb{R}^1\rightarrow \mathbb{R}^1$ are smooth functions.
 This work focuses on proposing and analyzing the a posteriori error estimates for the nonlinear stead-state PNP problem \eqref{MPNP-model}-\eqref{MPNP-model-bd0}.

 In general, there are two types of a posteriori error estimators, the gradient recovery-type a posteriori error estimator and the residual-type a posteriori error estimator. Compared with the residual-type a posteriori error estimator, the recovery-type a posteriori error estimator based on the gradient recovery operator is simpler in implementation. In this paper, by using a local averaging operator which is an extension of the gradient recovery operator, we derive a local averaging type a posteriori error estimates for the nonlinear PNP problem \eqref{MPNP-model}-\eqref{MPNP-model-bd0}.
 Then the global upper bounds and the local lower bounds of the error estimators are obtained for both the electrostatic potential and concentrations. A corresponding adaptive finite element algorithm is designed for the nonlinear PNP equations. Numerical experiments verify the efficiency and reliability of the error estimators derived in this paper.

 The rest of this paper is organized as follows. In section \ref{sec2}, the basic notations for Sobolev space and some useful preliminary results for the finite element approximation are introduced. In section \ref{sec-posteriori}, we present the global upper bounds and the local lower bounds both for the electrostatic potential and concentrations. Based on the a posterior error estimators, a corresponding adaptive finite element algorithm is also proposed in this section. In section \ref{sec-Numer-experiments}, numerical experiments are reported to support our theoretical analysis. Finally, in section \ref{sec-conclusion}, some conclusions are presented.

\setcounter{equation}{0}
\section{Preliminaries} \label{sec2}

\noindent
 In this section, we shall describe some basic notations and assumptions. Let $\Omega$ be a bounded domain in $\mathbb{R}^2$. For the integer $k \ge 0$ and $1 \le p \le \infty$, let $W^{k,p}(\Omega)$ be the Sobolev space with norm (see, e.g., \cite{R.A.Adams1975,S.C.Brenner2002}),
 \begin{equation*}
  \|u\|_{W^{k, p}}=\left\{\begin{array}{l}
   {\left(\sum\limits_{|\beta| \leq k} \int_{\Omega}\left|D^{\beta} u\right|^{p} \mathrm{d} x\right)^{\frac{1}{p}}, \quad {\rm for} ~ 1 \leq p<\infty}, \\
   {\sum\limits_{|\beta| \leq k} \operatorname{ess} \sup\limits _{\Omega}\left|D^{\beta} u\right|}, \quad\quad\quad~ {{\rm for}~ p=\infty},
  \end{array}\right.
 \end{equation*}
 where
 $$ D^\beta = \frac{\partial^{|\beta|}}{\partial x_1^{\beta_1}\partial x_2^{\beta_2}}, ~~~ |\beta|=\sum\limits_{i=1}^2\beta_i, $$
 for the multi-index $\beta=(\beta_1,\beta_2)$, $\beta_i \ge 0$, $1\le i \le 2$. For $p=2$, denote by $H^k(\Omega):=W^{k,2}(\Omega)$ and $ H_0^1(\Omega)=\{v|v\in H^1(\Omega):v|_{\partial \Omega}=0\}$, where $v|_{\partial \Omega}$ is in the sense of trace, $\|\cdot\|_{k,\Omega}:=\|\cdot\|_{k,2,\Omega}$ with the expression that $\|\cdot\|_0$ and $(\cdot,\cdot)$ denote the norm and inner product in $L^2$, respectively, and $\|\cdot\|_{0,\infty}:=\|\cdot\|_{L^\infty}$. Throughout this paper, we shall use $C$ denote a generic positive constant which may stand for different values at its different occurrences and are independent of the mesh parameters.

 Let $\mathcal{T}^h=\{\tau\}$ be a shape-regular simplices of $\Omega$ with mesh size $h = \max_{\tau\in\mathcal{T}^h } \{h_\tau\}$, where $h_\tau$ is the diameter of the elements $\tau$. Denote by $\partial \mathcal{T}^h$ the set of all surfaces  of simplices, $\partial^2 \mathcal{T}^h$ the set of all vertices of $\mathcal{T}^h$ and $\Lambda=\partial^{2}\mathcal{T}^{h}\backslash\partial\Omega$. We define the linear finite element space
 \begin{equation}\label{FEM-space}
  S^h = \{v\in H^1(\Omega ): v|_{\tau} \in \mathcal{P}^1(\tau),
  \forall \tau\in\mathcal{T}_h \},~~ S_0^h=S^h \cap H_0^1(\ome),
 \end{equation}
 where $\mathcal{P}^1 (\tau)$ is the space of linear polynomials on $\tau$. Let $\{\varphi_z:z\in\partial^2 \mathcal{T}^h\}\subset S^h$ be the standard nodal basis functions of $S^h$, namely,
 $$\varphi_{z_1}(z_2)=\delta_{z_1z_2}, ~~ \forall z_1,z_2 \in \partial^2 \mathcal{T}^h, $$
 where $\delta$ is the Kronecker symbol. For given $z\in\partial^2\mathcal{T}^h$, $l\in\partial \mathcal{T}^h$ and $\tau\in \mathcal{T}^h$,  denote by
 \begin{equation}\label{def-omega-zl-tau}
  \omega_{z}=\mathop\cup\limits_{z\in\bar{\tau}}\tau, ~~ \omega_{l}=\mathop\cup\limits_{l\subset\bar{\tau}}\tau,  ~~ \omega_{\tau}=\mathop\cup\limits_{\bar{\tau}'\cap\bar{\tau} \ne \emptyset}\tau',
 \end{equation}
 where $\bar{\tau}$ is the closure of $\tau$.

\subsection{Cl\'ement interpolation and local averaging operator}

\noindent
 We need introduce two Cl\'ement-type interpolation operators $\pi_{h}$ and $\Pi_{h}$: $L^2(\Omega)\rightarrow S^h_{0}$, which are defined respectively by (cf. \cite{P.Cle1975,N.Yan2001})
 \begin{align*}
  &\pi_{h}v=\sum_{z\in\Lambda}\upsilon_{z}\varphi_{z}, ~~~~~~~\upsilon_{z}=\frac {(\upsilon,\varphi_{z})}{(\varphi_{z},1)}, ~~\forall\upsilon\in L^2(\Omega), \\
  &\Pi_{h}\upsilon=\sum_{z\in\partial^2\mathcal{T}^h}\upsilon^z\varphi_{z},
  ~~\upsilon^z=\sum_{j=1}^{J_z}\alpha_z^j\upsilon|_{\tau_z^j}(z),
  ~~\forall\upsilon\in L^2(\Omega),
 \end{align*}
 where $\varphi_{z}$ is the basis function, $\cup_{j=1}^{J_z}\tau_z^j=\omega_{z}$, $\sum_{j=1}^{J_z}\alpha_z^j=1$, and $\alpha_{z}^{j}\ge 0$. For instance, $\alpha_z^j=\frac{1}{J_z}$ or $\alpha_{z}^{j}=\frac{|\tau_z^j|}{|\omega_z|} $. It should be pointed out that $\upsilon|_{\tau_{z}^j}$ is understood in the sense of trace in $\tau_z^j$ here. For $\upsilon\in H_0^1(\Omega)$, there hold (see e.g., \cite{N.Yan2001,R.Ver1998,C.Carstensen1999})\\
 \begin{align} \label{PCle-interpolation1}
 \|\upsilon-\pi_{h}\upsilon\|_{0,\tau}
 &\le Ch_{\tau}\|\nabla\upsilon\|_{0,\omega_\tau}, ~~~~\forall\tau\in \mathcal{T}^h, \\ \label{PCle-interpolation2}
 \|\upsilon-\pi_{h}\upsilon\|_{0,l}
 &\le Ch_{l}^{1/2}\|\nabla\upsilon\|_{0,\omega_l},~~\forall l\in\partial\mathcal{T}^h,
  \\ \label{PCle-interpolation3}
 |\pi_{h}\upsilon|_{1,\tau}
 &\le C|\upsilon|_{1,\omega_{\tau}}, ~~~~~~~~~~~\forall\tau\in \mathcal{T}^h.
 \end{align}

 The local averaging operator $G_{h}: S^h_0\rightarrow S^h\times S^h$ is defined as follows (cf. e.g., \cite{N.Yan2001,O.C.Zie1992})
 \begin{equation} \label{Gh-operator-alpha}
  G_h\upsilon =\Pi_h\big(\alpha(x,\upsilon)\nabla\upsilon\big), \quad
  (\alpha(x,\upsilon)\nabla\upsilon)_{z}=\sum_{j=1}^{J_{z}} \alpha_{z}^{j}(\alpha(z,\upsilon(z))\nabla\upsilon)_{\tau_{z}^{j}}, ~~ \forall\upsilon\in S^h_{0}.
 \end{equation}
 By the definition of the local averaging operator $G_h$, a smoothened flux field $G_hv$ is then obtained from the flux field  ``$\alpha(\cdot,\upsilon)\nabla v$". Hence, the operator $G_h$ is also called flux recovery operator (cf. \cite{D.Mao2006}). Note that if the coefficient $\alpha(\cdot,\upsilon)\equiv 1$, then a so-called gradient recovery operator $\widetilde{G}_{h} : S_{0}^{h} \mapsto S^{h} \times S^{h}$ is defined by (see e.g., \cite{N.Yan2001,O.C.Zie1992,Z.Zhang2001})
 \begin{align}\label{Ghwh0-0}
  \widetilde{G}_{h} v=\sum_{z \in \partial^2\mathcal{T}^h}(\nabla v)_{z} \varphi_{z}, \quad(\nabla v)_{z}=\sum_{j=1}^{J_{z}} \alpha_{z}^{j}(\nabla v)_{\tau_{z}^{j}}, ~~ \forall v \in S_{0}^{h}.
 \end{align}
In this paper, we shall use the local averaging operator $G_h$ to derive the a posteriori error estimators for the nonlinear PNP problem \eqref{MPNP-model}-\eqref{MPNP-model-bd0}.

 At the end of the subsection, according to the definition of the operator $\widetilde{G}_h$ and the properties of the basis function, we can easily get the following lemma which shall be used in our later analysis.
\begin{lemma} \label{lem-Ghwh}
  Suppose $\widetilde{G}_h$ is defined by \eqref{Ghwh0-0}. For any $w_h\in S_0^h$,  there hold
  \begin{align}\label{Ghwh-0}
   \|\widetilde{G}_hw_h\|_0 \le C\|\nabla w_h\|_0,
  \end{align}
 and
  \begin{align}\label{Ghwh-infty}
   \|\widetilde{G}_hw_h\|_{0,\infty} \le C\|\nabla w_h\|_{0,\infty}.
  \end{align}
\end{lemma}


\subsection{A finite element approximation} \label{sec3}

\noindent
 In this subsection, we consider the finite element approximation for the nonlinear PNP problem \eqref{MPNP-model}-\eqref{MPNP-model-bd0}.

 Suppose the problem \eqref{MPNP-model}-\eqref{MPNP-model-bd0} has a solution $(p^i,\phi)$, where $\phi\in H^1_0(\Omega)\cap W^{1,\infty}(\Omega) $ and $p^i\in H^1_0(\Omega)\cap W^{1,p}(\Omega)$ for some $p>2$. For any $w\in H^1_0(\Omega)\cap W^{1,p}(\Omega)$, the linearized operator $\mathcal{L}$ at $w$ (namely, the Fr\'echet derivative of $\mathcal{L}$ at $w$) is then given by
 \begin{align*}
  \mathcal{L}'(w,\phi)\psi=-\mbox{div}\big(\alpha(\cdot,w)\nabla\psi+\big(\alpha_y(\cdot,w)\nabla w+\beta_y(\cdot,w) \big)\psi\big)+g_y(\cdot,w)\psi -\mbox{div}\big(\gamma_y(\cdot,w)\nabla\phi\big)\psi.
 \end{align*}
 Furthermore, if we denote
 \begin{align*}
  H'(w)\psi=-\mbox{div}\big(\alpha(\cdot,w)\nabla\psi+\big(\alpha_y(\cdot,w)\nabla w+\beta_y(\cdot,w) \big)\psi\big)+g_y(\cdot,w)\psi,
 \end{align*}
 then, the bilinear form (induced by $H'(w)$) is that
 \begin{equation} \label{Hawpsiv}
 a'(w;\psi,v)=\big(\alpha(\cdot,w)\nabla\psi+
    \big(\alpha_y(\cdot,w)\nabla w+\beta_y(\cdot,w) \big)\psi,
    \nabla v\big)+\big(g_y(\cdot,w)\psi,v\big).
 \end{equation}

 Our basic assumptions are, first of all, the exact solution $p^i$ of \eqref{MPNP-model} satisfies
 \begin{equation}\label{elliptic-condition}
 \xi^\text{T}\alpha(x,p^i)\xi\geq C^{-1}|\xi|^2, ~~\forall\xi\in \mathbb{R}^2,~~x\in \bar{\Omega},
 \end{equation}
 for some constant $C>0$ and, secondly, $\mathcal{L}'(p^i,\phi):H^1_0(\Omega)\rightarrow H^{-1}(\Omega)$ is an isomorphism. As a result of these assumptions, $p^i$ must be an isolated solution (cf. \cite{XZ2001}).

 Denote by
 \begin{align} \label{aform1}
 a(w,v)=\big(\alpha(\cdot,w)\nabla w+\beta(\cdot,w),
              \nabla v\big)+\big( g(\cdot,w),v \big),
 \end{align}
 and
 \begin{align} \label{b-a2form}
  b(w,\psi,v)=\big(\gamma(\cdot,w)\nabla\psi,\nabla v\big), ~~~ \widetilde{a}(w, v)=\big( \epsilon(x)\nabla w,\nabla v\big).
 \end{align}
 Then the weak forms of \eqref{MPNP-model}-\eqref{MPNP-model-bd0} are that: find  $p^i, i=1,2,\cdots,n$ and $\phi \in H_0^1(\Omega)$ such that
 \begin{equation}\label{weak-p}
  a(p^i,v)+b(p^i,\phi,v)=0,  ~~ \forall v\in H^1_0(\Omega),  ~i=1,2,\cdots,n,
 \end{equation}
 \begin{equation}\label{weak-phi}
   \widetilde{a}(\phi,w) = \big(\widetilde{f}(p^i),w), ~~ \forall w\in  H^1_0(\Omega),
 \end{equation}
 where $\widetilde{f}(p^i):= \widetilde{f}(p^1,p^2,\cdots,p^n) = \sum\limits_{i=1}^nq^ip^i + f$.
 Similarly, from now on, we use $(p^i,\phi)$ to denote $(p^1,p^2$,$\cdots,p^n,\phi)$ for similicity.

 The corresponding finite element discretizations for \eqref{weak-p}-\eqref{weak-phi} are that: find $p_h^i, i=1, 2, \cdots, n$ and $\phi_h \in S_0^h$ such that
 \begin{equation}\label{fem-ph}
  a(p^i_h,v_h)+b(p^i_h,\phi_h,v_h)=0, ~~ \forall v_h\in S^h_0, ~i=1,2,\cdots,n,
 \end{equation}
 \begin{equation}\label{fem-phih}
  \widetilde{a}(\phi_h,w_h) = \big(\widetilde{f}(p_h^i),w_h), ~~~~~~~ \forall w_h\in S^h_0,
 \end{equation}
  where $\widetilde{f}(p_h^i) = \sum\limits_{i=1}^nq^ip_h^i + f$.

In the later analysis, we need the following identity. For any $w,\psi,\widetilde{w},\widetilde{\psi},v \in H^1_0(\ome)$, define the remainder
\begin{equation}\label{Remainder}
 R(w,\psi,\widetilde{w},\widetilde{\psi},v) = a(\widetilde{w},v)+b(\widetilde{w},\widetilde{\psi},v)-a(w,v)-b(w,\psi,v)-a'(w;\widetilde{w}-w,v).
\end{equation}
 By using the similar arguments in \cite{XZ2001}, we have the following lemma.
 \begin{lemma}\label{lem-Remainder}
 Let $(p^i,\phi)$ be the solution to \eqref{weak-p}-\eqref{weak-phi}. Then the finite element approximation $(p^i_h, \phi_h)$ is the  solution to \eqref{fem-ph} if and only if
 \begin{equation}\label{lem-R0}
  a'(p^i;p^i-p^i_h,v_h)=R(p^i,\phi,p^i_h,\phi_h,v_h), ~~\forall v_h\in S_0^h.
 \end{equation}
 Moreover, for any $w,\psi,\widetilde{w},\widetilde{\psi},v \in H^1_0(\Omega)$, if
 $\psi\in W^{1,\infty}(\Omega)$ and $\gamma(\cdot,\widetilde{w}) \in L^{\infty}(\Omega)$ satisfies
 \begin{equation} \label{gamma-chi-w}
 \|\gamma(\cdot,\widetilde{w})-\gamma(\cdot,w)\|_0 \le C\|\widetilde{w}-w\|_0,
 \end{equation}
 then the remainder $R$ satisfies
 \begin{align} \label{Rerror2}
  \big|R(w,\psi,\widetilde{w},\widetilde{\psi},v)\big|
  &\le C\big(\|w-\widetilde{w}\|_{1,3}\|w-\widetilde{w}\|_{1}  +\|\nabla\widetilde{\psi}-\nabla\psi\|_0
         +\|\widetilde{w}-w\|_{0}\big)\|v\|_{1}.
 \end{align}
\end{lemma}

\begin{proof} Taking $w=p^i$, $\psi=\phi,\widetilde{w}=p_h^i$, $\widetilde{\psi}=\phi_h$ and $v=v_h$ in (\ref{Remainder}) and from (\ref{weak-p}) and (\ref{fem-ph}), it yields (\ref{lem-R0}).

 Now we turn to show (\ref{Rerror2}). Let $\eta(t)=a(w+t(\widetilde{w}-w),v)$. Since
 \begin{equation*}
  \eta(1)=\eta(0)+\eta '(0)+\frac{1}{2}\int_0^1\eta''(t)(1-t)dt,
 \end{equation*}
 we have
 \begin{equation*}
	a(\widetilde{w} ,v)= a(w,v)+a'(w;\widetilde{w} -w ,v)
                         + \widetilde{R}(w,\widetilde{w} ,v),
 \end{equation*}
 where $\widetilde{R}(w,\widetilde{w},v) = \frac{1}{2}\int_0^1\eta''(t)(1-t)dt$. Compared with \eqref{Remainder}, it apparently shows that
 \begin{align} \nonumber
  |R(w,\psi,\widetilde{w},\widetilde{\psi},v)|
  &= |b(\widetilde{w},\widetilde{\psi},v)-b(w,\psi,v)+\widetilde{R}(w,\widetilde{w},v)| \\
  & \le |b(\widetilde{w},\widetilde{\psi},v)-b(w,\psi,v)|
       + |\widetilde{R}(w,\widetilde{w},v)|.
  \label{RB}
 \end{align}
 For the first term on the right-hand side in \eqref{RB}, by $\psi\in W^{1,\infty}(\Omega)$, $\gamma(\cdot,\widetilde{w})\in L^{\infty}(\Omega)$ and \eqref{gamma-chi-w}, we get
 \begin{align}\label{BB}
 \big|b(\widetilde{w},\widetilde{\psi},v)&-b(w,\psi,v)\big|
 =\big|\big(\gamma(\cdot,\widetilde{w})\nabla\widetilde{\psi}
            -\gamma(\cdot,w)\nabla\psi,\nabla v\big)\big| \notag \\
 &= \big|\big(\gamma(\cdot,\widetilde{w})(\nabla\widetilde{\psi}-\nabla\psi),\nabla v\big)
   +\big((\gamma(\cdot,\widetilde{w})-\gamma(\cdot,w))\nabla\psi,\nabla v\big)\big|
   \notag \\
 &\le C\big(\|\gamma(\cdot,\widetilde{w})\|_{0,\infty}\|\nabla\widetilde{\psi}-\nabla\psi\|_0
   +\|\nabla\psi\|_{0,\infty}\|\gamma(\cdot,\widetilde{w})
   -\gamma(\cdot,w)\|_{0} \big)\|v\|_1 \notag \\
 &\le C\big(\|\nabla\widetilde{\psi}-\nabla\psi\|_0
      +\|\widetilde{w}-w\|_{0} \big)\|v\|_1.
\end{align}

 On the other hand, from Lemma 3.1 in \cite{XZ2001}, the second term on the right-hand side in \eqref{RB} can be bounded by
 \begin{align}\label{rwr}
  |\widetilde{R}(w,\widetilde{w},v)|
  &\le C_\lambda \big(\|w-\widetilde{w}\|_{1,3}\|w-\widetilde{w}\|_{1}+(\|\nabla w\|_{0,p}
       +\|\nabla\widetilde{w}\|_{0,p})\|w-\widetilde{w}\|_{1,3}\|w-\widetilde{w}\|_{1} \big) \|v\|_1 \notag \\
  &\le C \|w-\widetilde{w}\|_{1,3}\|w-\widetilde{w}\|_{1}\|v\|_1,
 \end{align}
 where $C_\lambda=\max\{ |\alpha_y|,|\alpha_{yy}|,|\beta_{yy}|,|g_{yy}| \}$ on $ \bar{\Omega} \times [-\lambda,\lambda]$.

 Inserting \eqref{BB} and \eqref{rwr} into \eqref{RB}, then we complete the proof.  $\hfill\Box$
\end{proof}

\begin{lemma}\label{lem-phih-infty-new1}
 Let $(p^i, \phi)$ and $(p_h^i, \phi_h)$ be the solutions of \eqref{weak-p}-\eqref{weak-phi} and \eqref{fem-ph}-\eqref{fem-phih}, respectively. If
 $f\in H^1(\Omega)$, then we have
 \begin{equation}\label{phih-infty-new0}
 \|\nabla \phi_h\|_{0,\infty} \le C.
 \end{equation}
\end{lemma}

\begin{proof}
 From \eqref{fem-phih}, we know that $\phi_h$ is the finite element approximation to the solution of the following problem
 \begin{equation}\label{phih-infty-new1}
 -\nabla\cdot \big( \epsilon(x)\nabla\phi \big) = \sum\limits_{i=1}^nq^ip_h^i + f.
 \end{equation}
 Hence, by Gagliardo--Nirenberg--Sobolev inequality (see \cite{L.C.Evans1998}) and the regularity estimate (see \cite{Y.Chen1998}), we have
 \begin{align}  \nonumber
  \|\nabla\phi_h\|_{0,\infty}
  &\le C\|\phi\|_{1,\infty}\le C\|\phi\|_{2,4} \le C
       \|\sum\limits_{i=1}^n q^ip_h^i + f\|_{0,4} \\  \notag
  &\le C \big(\|\sum\limits_{i=1}^n \|p_h^i\|_{0,4} + \|f\|_{0,4}\big) \\  \notag
  &\le C \big(\|\sum\limits_{i=1}^n \|p_h^i\|_{1,2} + \|f\|_{1,2}\big) \\  \notag
  &\le C, \label{phi-infty-pro}
 \end{align}
 where we have used the assumption that $p_h^i$ and $f\in H^1(\ome)$. Thus, we finish the proof of Lemma \ref{lem-phih-infty-new1}.  $\hfill\Box$
\end{proof}

 In \cite{yanglu2013}, the a priori error estimate is shown for the potential as follows.
\begin{lemma} \label{the-phi-femH1}
 \cite{yanglu2013} Let $(p^i, \phi)$ and $(p_h^i, \phi_h)$ be the solutions of \eqref{weak-p}-\eqref{weak-phi} and \eqref{fem-ph}-\eqref{fem-phih}, respectively.  If $\phi\in H^2(\ome)$, then there holds
 \begin{equation} \label{phi-H1}
  \|\phi-\phi_h\|_1 \le C (h+\sum_{i=1}^n \|p^i-p^i_h\|_0).
 \end{equation}
\end{lemma}

 In the later analysis, we also need the following lemmas.
\begin{lemma}\label{lemm3.2} \cite{XZ2001,Xu1996}
  If $h \ll 1$ and $p^i$ is the solution of \eqref{weak-p}-\eqref{weak-phi}, then
  \begin{equation}\label{yin1}
    {\left\| w_h \right\|_{1}} \le C\mathop {\sup }\limits_{\varphi\in S_0^h} \frac{{a'({p^i},w_h,\varphi )}}{{{{\left\| \varphi  \right\|}_{1 }}}},\quad \forall w_h \in S_0^h.
 \end{equation}
\end{lemma}

\begin{lemma}\label{lem-p-Rhp} \cite{XZ2001}
  Let $R_h:H^1_0(\Omega)\rightarrow S^h_0$ be defined by
  \begin{equation} \label{p-Rhp-0}
	a'(p^i;p^i-R_hp^i,v_h)=0, ~~ \forall v_h\in S^h_0.
  \end{equation}
  If $p^i\in H_0^1(\Omega)\cap H^{2}(\Omega)$, then
  \begin{equation}\label{yin30}
   \|p^i - R_hp^i\|_{1,t}\le Ch^{1+2(1/t-1/2)}\|p^i\|_{2}, ~~ t\ge 2,
  \end{equation}
 and
 \begin{equation}\label{yin301}
  \|R_hp^i\|_{1,p}\leq C\|p^i\|_{2}.
 \end{equation}	
\end{lemma}

 Furthermore, by using Lemmas \ref{lemm3.2}, \ref{lem-p-Rhp} and the similar arguments as in \cite{XZ2001}, we have the the following lemma.
\begin{lemma}\label{lem-key}
 Let $(p^i, \phi)$ and $(p_h^i, \phi_h)$ be the solutions of \eqref{weak-p}-\eqref{weak-phi} and \eqref{fem-ph}-\eqref{fem-phih}, respectively.
 If $\phi,~p^i\in H_0^1(\Omega)\cap H^{2}(\Omega)$, $\|p^i-p^i_h\|_0\leq C h$
 and $h \ll 1$, then
 \begin{equation}\label{yin32}
  \|p^i_h-R_hp^i\|_{1}\le Ch.
 \end{equation}
\end{lemma}

\begin{proof}
 For any $\chi \in S^h_0$, let $\Phi:S^h_0\rightarrow S^h_0$ be defined by
 \begin{equation}\label{A'}
  a'(p^i;\Phi(\chi),v_h)=a'(p^i;p^i,v_h)-R(p^i,\phi,\chi,\phi_h,v_h), ~~\forall v_h\in S^h_0.
 \end{equation}
 Obviously, $\Phi$ is continuous. Define
 \begin{equation} \label{sh-0}
  \mathcal{B}=\{v_h\in S^h_0: \|v_h-R_hp^i\|_1\le Ch, ~ \|v_h\|_{1,p}\leq C(1 + \|p^i\|_2)\}.		
 \end{equation}
 If $\Phi(\mathcal{B})\subset \mathcal{B}$, then by Brouwer's fixed point theorem, there exists a fixed point $p^i_h\in\mathcal{B}$ such that $\Phi(p^i_h)= p^i_h$, which combining with \eqref{A'} and Lemma \ref{lem-Remainder} yields that
 $p_h^i$ is the finite element solution of \eqref{fem-ph} and \eqref{yin32} holds from $p^i_h\in\mathcal{B}$. Hence, in order to obtain \eqref{yin32}, we only need to show $\Phi(\mathcal{B})\subset \mathcal{B}$.

 For any $\chi\in B$, by \eqref{p-Rhp-0} and \eqref{A'}, we get
 \begin{equation*}
  a'(p^i;\Phi(\chi)-R_hp^i,v_h)=R(p^i,\phi,\chi,\phi_h,v_h), ~~\forall v_h\in S^h_0.
 \end{equation*}
 Since $\Phi(\chi)-R_hp^i \in S_0^h$, by Lemma \ref{lemm3.2}, it yields
 \begin{equation*}
  \|\Phi(\chi)-R_hp^i\|_{1} \le C\mathop{\sup }\limits_{\varphi\in S_0^h} \frac{{a'({p^i};\Phi(\chi)-R_hp^i,\varphi)}}{{{{\left\| \varphi \right\|}_{1}}}} \le C\mathop {\sup }\limits_{\varphi  \in S_0^h(\Omega )} \frac{{|R(p^i,\phi,\chi,\phi_h,\varphi)|}}{{{{\left\| \varphi \right\|}_{1}}}}.
 \end{equation*}
 From \eqref{Rerror2}, \eqref{phih-infty-new0} and $\chi\in \mathcal{B}$, there holds
 \begin{align}
  \|\Phi(\chi)-R_hp^i\|_{1} \le C\big(\|p^i-\chi\|_{1,3}\|p^i-\chi\|_{1}
   +\|\nabla\phi-\nabla\phi_h\|_0 + \|p^i-\chi\|_0\big).
 \end{align}
 Then from \eqref{phi-H1}, we have
 \begin{align}\label{yyphi-Rh}
  \|\Phi(\chi)-R_hp^i\|_{1}
  \le C \big( {\|p^i-\chi\|_{1,3}\|p^i-\chi\|_{1} + h +\|p^i-\chi\|_0}\big).
 \end{align}
 By inverse inequality and $\chi\in \mathcal{B}$, from \eqref{sh-0}, we get
 \begin{equation}\label{P'}
  \|R_hp^i-\chi\|_{1,3}\le Ch^{2(1/3-1/2)}\|R_hp^i-\chi\|_1
  \le Ch^{2/3}.
 \end{equation}

 In addition, combining \eqref{yin30} and \eqref{P'}, it follows that
 \begin{align*}
  \|p^i-\chi\|_{1,3}\|p^i-\chi\|_1
  &\le C\big(\|p^i-R_hp^i\|_{1,3}
     + \|R_hp^i-\chi\|_{1,3}\big)
     \big(\|p^i-R_hp^i\|_{1} + \|R_hp^i-\chi\|_{1}\big) \\
  &\le Ch^{5/3}.
 \end{align*}
 Hence, from \eqref{yyphi-Rh}, we have
 \begin{align}\label{P''}
  \|\Phi(\chi)-R_hp^i\|_1
  &\le C \Big( h^{5/3} + h + \|p^i-\chi\|_0 \Big) \notag \\
  &\le C \big( h + \|p^i-\chi\|_0 \big).
 \end{align}
 Then from \eqref{yin301} and \eqref{P''}, we deduce that
 \begin{align}\label{tang1}
  \|\Phi(\chi)\|_{1,p}
  &\le \|\Phi(\chi)-R_hp^i\|_{1,p}+\|R_hp^i\|_{1,p}  \notag \\
  &\le Ch^{2(1/p-1/2)}\|\Phi(\chi)-R_hp^i\|_{1}+\|R_hp^i\|_{1,p} \notag\\
  &\le C \Big( h^{2(1/p-1/2)}\big( h + \|p^i-\chi\|_0 \big)
         + \|p^i\|_2 \Big).
 \end{align}
 For the term $\|p^i-\chi\|_0$, from $\chi \in \mathcal{B}$, \eqref{yin30} and \eqref{sh-0}, we have
 \begin{equation}\label{tang2}
 \|p^i-\chi\|_0\le \|p^i-R_hp^i\|_0+\|R_hp^i-\chi\|_0\le Ch.
 \end{equation}
 Substituting \eqref{tang2} into \eqref{P''} and \eqref{tang1}, respectively, it easily yields
 \begin{align*}
  \|\Phi(\chi)-R_hp^i\|_1 \le C h
 \end{align*}
 and
 \begin{align*}
  \|\Phi(\chi)\|_{1,p}\le C (1 + \|p^i\|_2).
 \end{align*}
 Hence $\Phi(\mathcal{B})\subset \mathcal{B}$. This completes the proof of \eqref{yin32}.    $\hfill\Box$
\end{proof}

\setcounter{equation}{0}
\setcounter{lemma}{0}
\section{A posteriori error estimates} \label{sec-posteriori}

\noindent
 In this section, we first present the a posteriori error estimates including the global upper bounds and the local lower bounds for the nonlinear PNP equations \eqref{MPNP-model} with boundary conditions \eqref{MPNP-model-bd0}. Then a typical finite element adaptive algorithm is developed based on the a posteriori error analysis.

\subsection{Upper bound} \label{subsec-Upper}

\noindent
 In this subsection, we shall derive the global upper bounds of the a posteriori error indicators for both the electrostatic potential and concentrations.

 First, the upper bound of $\|\nabla(\phi-\phi_h)\|_{0,\Omega}$ is presented as follows.
\begin{theorem}\label{theo-upper-phi}
 Let $(p^i, \phi)$ and $(p_h^i, \phi_h)$ be the solutions of \eqref{weak-p}-\eqref{weak-phi} and \eqref{fem-ph}-\eqref{fem-phih}, respectively. There holds
 \begin{equation}\label{phi-upper}
 \|\nabla(\phi-\phi_h)\|_{0,\Omega} \le C\Big(\eta_{\phi}(p^i_{h},\phi_{h}) +\sum\limits^{n}_{i=1}\|p^i-p^i_{h}\|_{0,\Omega} \Big),
 \end{equation}
 where
 \begin{align*}
 &\eta_{\phi}(p^i_{h},\phi_{h}) = \sum_{\tau\in\mathcal{T}^h}
        \Big( h_\tau\|R_{1h}(p^i_{h},\phi_{h})\|_{0,\tau}
           + \|D_{h}(\phi_{h})\|_{0,\tau} \Big), \\
 &R_{1h}(p^{i}_{h},\phi_{h})=\sum_{i=1}^nq^ip^i_{h}
           +{\rm div}(G_{h}\phi_h) + f, \\
 &D_{h}(\phi_{h})= G_{h}\phi_{h}-\epsilon(x)\nabla\phi_{h}.
 \end{align*}
\end{theorem}

\begin{proof}
 For any $w\in H_{0}^1(\Omega)$, $\chi\in S_0^h$, it follows from \eqref{weak-phi} and \eqref{fem-phih} that
 \begin{align}
  \widetilde{a}(\phi-\phi_{h},w)
  &= \widetilde{a}(\phi,w) - \widetilde{a}(\phi_{h},w)
     = (\widetilde{f}(p^i),w) - \widetilde{a}(\phi_{h},w)  \notag \\
  &= (\widetilde{f}(p^i),w-\chi) + (\widetilde{f}(p^i),\chi) - \widetilde{a}(\phi_{h},w) \notag \\
  &= (\widetilde{f}(p^i),w-\chi)
     + (\widetilde{f}(p^i)-\widetilde{f}(p_h^i),\chi)
     + (\widetilde{f}(p_h^i),\chi) - \widetilde{a}(\phi_{h},w-\chi) - \widetilde{a}(\phi_{h},\chi) \notag \\
  &= \big(\sum_{i=1}^nq^ip^i+f,w-\chi\big)-\big(\sum_{i=1}^nq^i(p^i-p^i_h),w-\chi\big) \notag \\
  &\quad   +\big(\sum_{i=1}^nq^i(p^i-p^i_h),w\big) - \widetilde{a}(\phi_{h},w-\chi)  \notag \\
  &= \big(\sum_{i=1}^nq^ip_h^i+f,w-\chi\big)
      + \big(\sum_{i=1}^nq^i(p^i-p^i_h),w\big)
      - \widetilde{a}(\phi_{h},w-\chi).
    \label{upphi-1}
 \end{align}
 By Green's formula, we rewrite the third term on the right-hand side of \eqref{upphi-1} as follows
 \begin{align}\label{upphi-2}
  - \widetilde{a}(\phi_{h},w-\chi)
  &= -\int_\Omega \epsilon(x)\nabla\phi_h\cdot\nabla(w-\chi) \notag \\
  &= \int_\Omega (G_h\phi_h - \epsilon(x)\nabla\phi_h)\cdot\nabla(w-\chi)
    - \int_\Omega G_h\phi_h\cdot\nabla(w-\chi) \notag \\
 &=  \int_\Omega (G_h\phi_h - \epsilon(x)\nabla\phi_h)\cdot\nabla(w-\chi)
    + \sum_{\tau\in\mathcal{T}^h}\int_{\tau} {\rm div}(G_h\phi_h)(w-\chi).
 \end{align}
 Substituting \eqref{upphi-2} into \eqref{upphi-1}, we get
 \begin{align}
 \widetilde{a}(\phi-\phi_{h},w)
 &=\sum_{\tau\in\mathcal{T}^h}\int_{\tau}R_{1h}(p^{i}_{h},\phi_{h})(w-\chi)
         +\sum_{\tau\in\mathcal{T}^h}\int_{\tau} D_{h}(\phi_{h})\cdot
        \nabla(w-\chi) \notag \\
 &\quad +\sum_{\tau\in\mathcal{T}^h}\int_{\tau}\big(\sum_{i=1}^nq^i(p^i-p^i_h)\big)w  \label{upphi-3-0}\\
 & \le C\sum_{\tau\in\mathcal{T}^h}\Big(\big\|R_{1h}(p^{i}_{h},\phi_{h})\|_{0,\tau}
               \|w-\chi\|_{0,\tau}
      +\|D_{h}(\phi_{h})\|_{0,\tau}\|\nabla(w-\chi)\|_{0,\tau}
        \notag \\
 &\quad +\sum^{n}_{i=1}\|p^i-p^i_{h}\|_{0,\tau}\|w\|_{0,\tau} \Big),
   \label{upphi-3}
 \end{align}
 where
 \begin{align*}
  R_{1h}(p^{i}_{h},\phi_{h})=\sum_{i=1}^nq^ip^i_{h}
           +{\rm div}(G_{h}\phi_h) + f, ~~~
  D_{h}(\phi_{h})= G_{h}\phi_{h} - \epsilon(x)\nabla\phi_{h}.
 \end{align*}
 Taking $\chi=\pi_hw$ in \eqref{upphi-3} and using Cl\'ement interpolation estimates \eqref{PCle-interpolation1} and \eqref{PCle-interpolation3}, it yields
 \begin{align}
 \widetilde{a}(\phi-\phi_{h},w)
 & \le C\sum_{\tau\in\mathcal{T}^h} \Big(\big(h_\tau\|R_{1h}(p^{i}_{h},\phi_{h})\|_{0,\tau}
      + \|D_{h}(\phi_{h})\|_{0,\tau} \big)\|\nabla w\|_{0,\omega_\tau} \notag\\
 &\quad  + \sum^{n}_{i=1}\|p^i-p^i_{h}\|_{0,\tau}\|w\|_{0,\tau} \Big).
  \label{upphi-4}
 \end{align}
 Then the desired result \eqref{phi-upper} can be easily obtained by taking $ w=\phi-\phi_{h}$ in \eqref{upphi-4}.   $\hfill\Box$
\end{proof}

Now we turn to present the upper bound of $\|\nabla(p^i-p_h^i)\|_{0,\Omega}$.
 First, we need the following lemma.
 \begin{lemma} \label{lemma-w-sup} \cite{XZ2001,Y.Yang2006}
 Suppose $p^i$ is an isolated solution and the finite element solution $p_h^i$ is sufficiently close to the exact solution $p^i$ provided by $h\ll1$. Then for any $w\in H_0^1(\Omega)$, there holds
 \begin{equation} \label{lem-w}
  \|w\|_{1,\Omega} \le C \sup\limits_{v\in H_0^1(\Omega)} \frac{a'(p_h^i;w,v)}{\|v\|_{1,\Omega}}.
  \end{equation}
\end{lemma}

 The global upper bound for $\|\nabla(p^i-p_h^i)\|_{0,\Omega}$ is presented as follows.
\begin{theorem}\label{theo-upper-p}
 Let $(p^i, \phi)$ and $(p_h^i, \phi_h)$ be the solutions of \eqref{weak-p}-\eqref{weak-phi} and \eqref{fem-ph}-\eqref{fem-phih}, respectively. Suppose the finite element solution $p_h^i$ is sufficiently close to the exact solution $p^i$.
 If $\phi, ~p^i\in H_0^1(\ome)\cap H^2(\ome)$, $\|p^i-p^i_h\|_0\leq Ch$
 and $h\ll 1$, then there holds
 \begin{align} \label{upper-p-0}
  \|\nabla(p^i-p_h^i)\|_{0,\Omega} \le  C\Big(\eta_{p^i}(p_h^i,\phi_h) +\sum\limits^{n}_{i=1}\|p^i-p_h^i\|_{0,\Omega}\Big),
 \end{align}
 where
 \begin{align*}
  &\eta_{p^i}(p_h^i,\phi_h)
  =  \sum_{\tau\in\mathcal{T}^h}
       \Big( \|D_h(p_h^i)\|_{0,\tau}
         + \|\gamma(x,p_h^i)(\widetilde{G}_h\phi_h-\nabla\phi_h)\|_{0,\tau}
         + h_\tau \|R_{2h}(p_h^i,\phi_h)\|_{0,\tau} \Big) \notag \\
  &\quad\quad\quad\quad\quad\quad +\sum_{\tau\in\mathcal{T}^h}
         \Big( \|D_{h}(\phi_{h})\|_{0,\tau}
          + h_\tau\|R_{1h}(p^i_{h},\phi_{h})\|_{0,\tau} \Big), \\
 &D_h(p_h^i)= G_hp_h^i-\alpha(x,p_h^i)\nabla p_h^i, ~~~
  D_{h}(\phi_{h})= G_{h}\phi_{h}-\epsilon(x)\nabla\phi_{h}, \\
 &R_{1h}(p^{i}_{h},\phi_{h})=\sum_{i=1}^nq^ip^i_{h}
       +{\rm div}(G_{h}\phi_h)+f, \\
 &R_{2h}(p^{i}_{h},\phi_{h})= {\rm div}(G_hp_h^i)
       +{\rm div}\big(\beta(x,p_h^i)\big)-g(x,p_h^i)
       +{\rm div}\big(\gamma(x,p_h^i)\widetilde{G}_h\phi_h\big).
 \end{align*}
\end{theorem}

\begin{proof}
 For any $v\in H_{0}^1(\Omega)$, $\chi\in S_0^h$, first, by \eqref{fem-ph}, we have
 \begin{align}\label{upper-p-1}
  -a(p_h^i,v)
  &=-\big(a(p_h^i,v-\chi)+a(p_h^i,\chi)\big) \notag \\ &=-a(p_h^i,v-\chi)+b(p_h^i,\phi_h,\chi) \notag \\
  &=-\big(\alpha(x,p_h^i)\nabla p_h^i+\beta(x,p_h^i),\nabla(v-\chi)\big)
         -\big(g(x,p_h^i),v-\chi\big)+b(p_h^i,\phi_h,\chi) \notag \\
  &=-\big(\alpha(x,p_h^i)\nabla p_h^i-G_hp_h^i,\nabla(v-\chi)\big)
         -\big(G_hp_h^i+\beta(x,p_h^i),\nabla(v-\chi)\big) \notag \\
  &\quad -\big(g(x,p_h^i),v-\chi\big)+b(p_h^i,\phi_h,\chi) \notag \\
  &=\sum_{\tau\in\mathcal{T}^h}\int_{\tau}\big(G_hp_h^i-\alpha(x,p_h^i)
          \nabla p_h^i\big)\cdot\nabla(v-\chi)
         +\sum_{\tau\in\mathcal{T}^h}\int_{\tau}
         \mbox{div}\big(G_hp_h^i+\beta(x,p_h^i)\big)(v-\chi)
         \notag \\
  &\quad -\sum_{\tau\in\mathcal{T}^h}\int_{\tau}
         g(x,p_h^i)(v-\chi) + b(p_h^i,\phi_h,\chi) \notag \\
  &=\widetilde{A}(p_h^i,v,\chi) + b(p_h^i,\phi_h,\chi),
 \end{align}
 where
 \begin{align} \label{widetilde-A-phvchi}
 &\widetilde{A}(p_h^i,v,\chi)= \sum_{\tau\in\mathcal{T}^h}\int_{\tau}
           D_h(p_h^i)\cdot\nabla(v-\chi)
         +\sum_{\tau\in\mathcal{T}^h}\int_{\tau}R_h(p_h^i)(v-\chi), \\ \notag
 &D_h(p_h^i)= G_hp_h^i-\alpha(x,p_h^i)\nabla p_h^i,  ~~~
  R_h(p_h^i)={\rm div}(G_hp_h^i)
       +\mbox{div}\big(\beta(x,p_h^i)\big)-g(x,p_h^i).
 \end{align}
 On the other hand, by \eqref{weak-p} and taking $\widetilde{w}=p^i,\widetilde{\psi}=\phi,w=p_h^i,\psi=\phi_h$ in \eqref{Remainder}, we have
 \begin{align} \label{upper-p-2}
  a'(p_h^i;p^i-p_h^i,v)
  &=a(p^i,v)+b(p^i,\phi,v)-a(p_h^i,v)-b(p_h^i,\phi_h,v)
    -R(p_h^i,\phi_h,p^i,\phi,v) \notag\\
  &= -a(p_h^i,v)-b(p_h^i,\phi_h,v)-R(p_h^i,\phi_h,p^i,\phi,v).
  \end{align}
 Combining \eqref{upper-p-1}, \eqref{widetilde-A-phvchi} and \eqref{upper-p-2}, it yields
 \begin{align} \label{upper-p-3}
 a'&(p_h^i;p^i-p_h^i,v)
  = \widetilde{A}(p_h^i,v,\chi) + b(p_h^i,\phi_h,\chi-v)-R(p_h^i,\phi_h,p^i,\phi,v)
    \notag \\
 &=\widetilde{A}(p_h^i,v,\chi)+\big(\gamma(x,p_h^i)\nabla\phi_h,\nabla(\chi-v)\big)
       - R(p_h^i,\phi_h,p^i,\phi,v) \notag \\
 &=\widetilde{A}(p_h^i,v,\chi) +
         \big(\gamma(x,p_h^i)(\nabla\phi_h-\widetilde{G}_h\phi_h),\nabla(\chi-v)\big)
        + \big(\gamma(x,p_h^i)\widetilde{G}_h\phi_h,\nabla(\chi-v)\big)
        - R(p_h^i,\phi_h,p^i,\phi,v) \notag \\
 &=\widetilde{A}(p_h^i,v,\chi)+\sum_{\tau\in\mathcal{T}^h}\int_{\tau}
        \big(\gamma(x,p_h^i)(\widetilde{G}_h\phi_h-\nabla\phi_h)\big)\cdot\nabla(v-\chi)
        +\sum_{\tau\in\mathcal{T}^h}\int_{\tau}
          \mbox{div}\big(\gamma(x,p_h^i)\widetilde{G}_h\phi_h\big)(v-\chi)
          \notag \\
 &\quad\quad\quad\quad\quad\quad  - R(p_h^i,\phi_h,p^i,\phi,v) \notag \\
 &=\sum_{\tau\in\mathcal{T}^h}\int_{\tau}
       D_h(p_h^i)\cdot\nabla(v-\chi)
      +\sum_{\tau\in\mathcal{T}^h}\int_{\tau}\big(R_h(p_h^i)
      +{\rm div}\big(\gamma(x,p_h^i)\widetilde{G}_h\phi_h\big)\big)(v-\chi) \notag \\
 &\quad      +\sum_{\tau\in\mathcal{T}^h}\int_{\tau}
        \big(\gamma(x,p_h^i)(\widetilde{G}_h\phi_h-\nabla\phi_h)\big)\cdot\nabla(v-\chi)
        - R(p_h^i,\phi_h,p^i,\phi,v).
 \end{align}
 Then taking $\chi=\pi_hv$ in \eqref{upper-p-3}, by Cl\'ement interpolation \eqref{PCle-interpolation1} and \eqref{PCle-interpolation3}, we get
 \begin{align} \label{upper-p-4}
  a'(p_h^i;p^i-p_h^i,v)
  &\le C \sum_{\tau\in\mathcal{T}^h}
       \Big( \|D_h(p_h^i)\|_{0,\tau}
        + \|\gamma(x,p_h^i)(\widetilde{G}_h\phi_h-\nabla\phi_h)\|_{0,\tau} \notag \\
  &\quad + h_\tau \|R_h(p_h^i)+{\rm div}
         \big(\gamma(x,p_h^i)\widetilde{G}_h\phi_h\big)\|_{0,\tau} \Big) \|v\|_{1,\omega_\tau}
        + \big|R(p_h^i,\phi_h,p^i,\phi,v)\big|.
 \end{align}

 Now it remains to estimate the remainder $R(p_h^i,\phi_h,p^i,\phi,v)$. From \eqref{Rerror2}, we have
 \begin{align} \label{R-ph-phih-p-phi}
 \big|R(p_h^i,\phi_h,p^i,\phi,v)\big|
 \le C \big(\|p_h^i-p^i\|_{1,3}\|p_h^i-p^i\|_{1} +\|\nabla(\phi-\phi_h)\|_0
        +\|p^i-p_h^i\|_{0}\big)\|v\|_{1}.
 \end{align}
 We turn to estimate $\|p_h^i-p^i\|_{1,3}$ on the right-hand side of \eqref{R-ph-phih-p-phi}.

 By the inverse estimate, \eqref{yin30} and \eqref{yin32}, it yields
 \begin{align}\label{ph-pi13}
  \|p_h^i-p^i\|_{1,3}
  &\le C \big(\|p_h^i-R_hp^i\|_{1,3}+\|R_hp^i-p^i\|_{1,3}\big) \notag \\
  &\le C \big(h^{2(1/3-1/2)}\|p_h^i-R_hp^i\|_{1}
        + h^{1+2(1/3-1/2)}\|p^i\|_{2}\big)\notag \\
  &\le C h^{\frac{2}{3}}.
 \end{align}
 Substituting \eqref{ph-pi13} into \eqref{R-ph-phih-p-phi}, we have
 \begin{align} \label{R-phphihpphi-1}
  \big|R(p_h^i,\phi_h,p^i,\phi,v)\big|
  \le C \big( h^{\frac{2}{3}}\|p_h^i-p^i\|_{1}
       + \|\nabla(\phi-\phi_h)\|_0+\|p^i-p_h^i\|_{0} \big)\|v\|_{1}.
 \end{align}
  By \eqref{upper-p-4}, \eqref{R-phphihpphi-1} and using Lemma \ref{lemma-w-sup}, it follows that
 \begin{align} \label{upper-p-5}
  \|p^i-p_h^i\|_{1,\Omega}
  & \le C \sum_{\tau\in\mathcal{T}^h}
       \Big( \|D_h(p_h^i)\|_{0,\tau}
        + \|\gamma(x,p_h^i)(\widetilde{G}_h\phi_h-\nabla\phi_h)\|_{0,\tau} \notag \\
  &\quad + h_\tau \|R_h(p_h^i)+{\rm div}
         \big(\gamma(x,p_h^i)\widetilde{G}_h\phi_h\big)\|_{0,\tau} \Big)  \notag\\
  &\quad + C \big( h^{\frac{2}{3}}\|p_h^i-p^i\|_{1}
         + \|\nabla(\phi-\phi_h)\|_0+\|p^i-p_h^i\|_{0} \big).
 \end{align}
  Hence, choosing $h$ sufficiently small such that $Ch^{\frac{2}{3}}\ll1$, then we obtain
 \begin{align} \label{upper-p-5-1}
  \|p^i-p_h^i\|_{1,\Omega}
  & \le C \sum_{\tau\in\mathcal{T}^h}
       \Big( \|D_h(p_h^i)\|_{0,\tau}
         + \|\gamma(x,p_h^i)(\widetilde{G}_h\phi_h-\nabla\phi_h)\|_{0,\tau}
         + h_\tau \|R_{2h}(p_h^i,\phi_h)\|_{0,\tau} \Big) \notag \\
  &\quad + C \big( \|\nabla(\phi-\phi_h)\|_{0,\Omega}+\|p^i-p_h^i\|_{0} \big),
 \end{align}
 where
 \begin{align*}
 R_{2h}(p_h^i,\phi_h)
 &= R_h(p_h^i)+{\rm div}\big(\gamma(x,p_h^i)\widetilde{G}_h\phi_h\big) \\
 &= {\rm div}(G_hp_h^i)+{\rm div}\big(\beta(x,p_h^i)\big)
    -g(x,p_h^i) +{\rm div}\big(\gamma(x,p_h^i)\widetilde{G}_h\phi_h\big).
\end{align*}

 Then the desired result \eqref{upper-p-0} is completed by using \eqref{phi-upper} and \eqref{upper-p-5-1}. This completes the
 proof of Theorem \ref{theo-upper-p}.  $\hfill\Box$
\end{proof}

\begin{remark}\label{Remark-upper}
 Note that $\|p^i-p_h^i\|_{0,\Omega}$ is usually higher-order term compared with $\|\nabla(\phi-\phi_h)\|_{0,\Omega}$ and $\|\nabla(p^i-p_h^i)\|_{0,\Omega}$.
 From \eqref{phi-upper} and \eqref{upper-p-0},
  if $\|p^i-p_h^i\|_{0,\Omega}\le Ch\|\nabla(p^i-p_h^i)\|_{0,\Omega}$,
  then
 \begin{align*} 
  \|\nabla(p^i-p_h^i)\|_{0,\Omega} \le C\eta_{p^i}(p_h^i,\phi_h),\\
  \|\nabla(\phi-\phi_h)\|_{0,\Omega} \le C \Big(\eta_{\phi}(p^i_{h},\phi_{h}) + h\eta_{p^i}(p_h^i,\phi_h)\Big).
 \end{align*}

  Up to now, there is no relevant work on the $L^2$ norm error estimate for both the steady-state PNP equations ( \eqref{MPNP-model} with $\alpha=\beta=1$, $\gamma=p^i,~g=f_i(x)$) and the nonlinear PNP equations \eqref{MPNP-model}. It is difficult to derive the $L^2$ norm error estimate for $p^i$ by using the traditional duality arguments for the steady-state PNP equations. Recently, we present an optimal $L^2$ norm error estimate of the finite element approximation $p_h^i$ in \cite{shen2019} for a time-dependent PNP equations, but the arguments used in \cite{shen2019} can not successfully applied to the steady-state model because of the difference between the steady-state and time-dependent PNP equations. Although there is no theoretical results on the $L^2$ norm error estimate for the steady-state PNP equations, numerical examples including PNP equations for practical biological problems solved on irregular meshes show that $\|p^i-p_h^i\|_{0,\Omega}\le Ch\|\nabla(p^i-p_h^i)\|_{0,\Omega}$ holds (see Figs. 2, 3 and Figs. 5, 6 in our work \cite{yanglu2013}, where Figs. 5, 6 presents the results for a practical biological problem).
\end{remark}

 In the next subsection, we will present the local lower bounds of the a posteriori error indicators for both the electrostatic potential and concentrations.

\subsection{Lower bound} \label{subsec-Lower}

\noindent
 Now, we study the lower bounds of the a posteriori error indicators for both the electrostatic potential and concentrations. We need further some assumptions for the coefficients in \eqref{MPNP-model}. Suppose that $\alpha(x,y)\in W^{2,\infty}(\Omega\times(-\lambda,\lambda))$, $\beta(x,y)\in\big(H^2(\Omega\times(-\lambda,\lambda))\big)^2$, $g(x,y)\in H^1(\Omega\times(-\lambda,\lambda))$, and $\epsilon(x)\in W^{2,\infty}(\omega_\tau)$, where $l\in \partial \mathcal{T}^h$, $l\not\subset \partial\Omega$ and $\omega_l$ is defined in \eqref{def-omega-zl-tau}. In addition, we assume that $\gamma(x,y)$ is a linear function with respect to the second variable $y$ (actually, $\gamma(x,p^i)=q^ip^i$ in the practical problems, where $q^i$ is a constant). We also need the assumption that there exist positive constants $\alpha_0$, $\alpha_1$, $\epsilon_0$ and $\epsilon_1$, such that
  \bea \label{boualp}\alpha_0 \le \alpha(x,y) \le \alpha_1
  \eea
  and
  \bea\label{bouepi}
  \epsilon_0 \le \epsilon(x) \le \epsilon_1.
  \eea

 Denote by $[g]_l$ the jump of $g$ across the surface $l\in \partial \mathcal{T}^h, l\not\subset\partial\Omega$, for example,
 \begin{equation} \label{alpha-jh}
 [\alpha(\cdot,v)\nabla v\cdot n_l]_l=\lim_{s\rightarrow 0^+}[\big(\alpha(\cdot,v)\nabla v\big)(x+sn_l)-\big(\alpha(\cdot,v)\nabla v\big)(x-sn_l)]\cdot n_l,
 \end{equation}
 and
 \begin{equation} \label{epsilon-jh}
  [\epsilon(x)\nabla v\cdot n_l]_l=\lim_{s\rightarrow 0^+}
  [\big(\epsilon(x)\nabla v\big)(x+sn_l)-\big(\epsilon(x)\nabla v\big)(x-sn_l)]\cdot n_l,
 \end{equation}
 where $n_l$ is the unit normal vector to $l$ and $v\in H_0^1(\Omega)$.

 Let $J_{h,l}(p_h^i)=[\alpha(x,p_h^i)\nabla p_h^i\cdot n_l]$ and $\widehat J_{h,l}(\phi_h)=[\epsilon(x)\nabla\phi_h\cdot n_l]$ represent the jumps of $p_h^i$ and $\phi_h$ across the surface $ l\in \partial \mathcal{T}^h$, $l\not\subset \partial\Omega$, respectively, where $p_h^i, ~ \phi_h \in S_0^h$. \\

 The following results will be used in our analysis for the lower bound.
\begin{lemma}\label{lem-lower}(cf. \cite{R.Ver1998,M.Ain2000})
 Let $\tau\in \mathcal{T}^h$ be a shape-regular mesh and $l\in \partial \mathcal{T}^h$. Then there exists $\mu_\tau:\mathcal{P}^1(\tau)\rightarrow H^1_{0}(\tau)$, such that for any $\upsilon\in \mathcal{P}^1(\tau)$, there hold
 \begin{align}\label{lem-lower-1}
  C^{-1}\|\mu_{\tau}\upsilon\|^2_{0,\tau} &\le \|\upsilon\|^2_{0,\tau} \le C(\upsilon,\mu_{\tau}\upsilon)_{\tau},
  \\  \label{lem-lower-2}
  \|\mu_{\tau}\upsilon\|_{1,\tau} &\le C h^{-1}_{\tau}\|\upsilon\|_{0,\tau}.
 \end{align}
 And there exists $\nu_{l}:\mathcal{P}^1(l)\rightarrow H^1_{0}(\omega_{l})$, such that for any $\upsilon\in \mathcal{P}^1(l)$, there hold
 \begin{align} \label{lem-lower-3}
  C^{-1}\|\nu_{l}\upsilon\|^2_{0,l} &\le \|\upsilon\|^2_{0,l} \le C(\upsilon,\nu_{l}\upsilon)_{l},
  \\ \label{lem-lower-4}
  \|\nu_{l}\upsilon\|_{0,\omega_{l}} &\le Ch^{\frac{1}{2}}_{\omega_{l}}\|\upsilon\|_{0,l},
  \\ \label{lem-lower-5}
  \|\nu_{l}\upsilon\|_{1,\omega_{l}} &\le C h^{-\frac{1}{2}}_{\omega_{l}}\|\upsilon\|_{0,l}.
 \end{align}
\end{lemma}

 From \eqref{upphi-1} and Lemma \ref{lem-lower}, we get the following result.
\begin{lemma}\label{lem-low-phi}
 Let $(p^i, \phi)$ and $(p_h^i, \phi_h)$ be the solutions of \eqref{weak-p}-\eqref{weak-phi} and \eqref{fem-ph}-\eqref{fem-phih}, respectively. For any $l\in \partial \mathcal{T}^h$, $l\not\subset\partial\Omega$, if $h_l\ll 1$, then there holds
 \begin{equation}\label{lem-low-phi-0}
  h^{\frac 1 2}_{\omega_{l}}\|\widehat J_{h,l}(\phi_{h})\|_{0,l}
  \le C\Big( \|\nabla(\phi-\phi_h)\|_{0,\omega_l} + h_{\omega_l}\sum\limits_{i=1}^n\|p^i-p_h^i\|_{0,\omega_l}\Big)
       +r_{\phi_h,\omega_l},
 \end{equation}
 where $\widehat J_{h,l}(\phi_h)=[\epsilon\nabla\phi_h\cdot n_l]$, and  $r_{\phi_h,\omega_l}\le Ch^2_{\omega_l}\big( |f|_{1,\omega_l}+\|\epsilon\|_{2,\infty,\omega_l}\|\phi_h\|_{1,\omega_l}  \big)$.
\end{lemma}

\begin{proof}
 For any $w\in H_{0}^1(\Omega)$, $\chi\in S_0^h$, $l\in \partial \mathcal{T}^h$ and $l\not\subset\partial\Omega$, from \eqref{upphi-1} and Green's formula, we get
 \begin{align}
  \widetilde{a}(\phi-\phi_{h},w)
  &= \big(\sum_{i=1}^nq^ip_h^i+f,w-\chi\big) + \big(\sum_{i=1}^nq^i(p^i-p^i_h),w\big)
      - \widetilde{a}(\phi_{h},w-\chi)  \notag \\
  &= \big(\sum_{i=1}^nq^ip_h^i+f,w-\chi\big) + \big(\sum_{i=1}^nq^i(p^i-p^i_h),w\big)
      - \int_\Omega \epsilon\nabla\phi_h\cdot\nabla(w-\chi) \notag \\
  &= \sum_{\tau}\int_{\tau}\widehat R_{1h}(p^{i}_{h},\phi_{h})(w-\chi)
          + \int_\Omega\Big(\sum_{i=1}^nq^i(p^i-p^i_h)\Big)w
          - \sum_{l\in\partial \mathcal{T}^h,l\not\subset \partial\Omega}\int_l \widehat J_{h,l}(\phi_h)(w-\chi),
    \label{upphi-000}
 \end{align}
 where
 \begin{align*}
  &\widehat R_{1h}(p^{i}_{h},\phi_{h})=
    \sum\limits_{i=1}^nq^ip_h^i
          +{\rm div}(\epsilon\nabla\phi_h) + f , ~~~
   \widehat J_{h,l}(\phi_h)
    =[\epsilon \nabla\phi_{h}\cdot n_{l}].
 \end{align*}
 Taking $\chi=0$ in \eqref{upphi-000}, for any $w\in H^1_{0}(\omega_l)$, we have
 \begin{align}
 \big(\widehat J_{h,l}(\phi_{h}),w\big)_l&=\sum_{l\in\partial \mathcal{T}^h,l\not\subset \partial\Omega}\int_l \widehat J_{h,l}(\phi_h) w \notag\\
 &=\big(\widehat R_{1h}(p^{i}_{h},\phi_{h}),w\big)_{\omega_l} - \widetilde{a}(\phi-\phi_{h},w)_{\omega_l}
        +\big(\sum_{i=1}^nq^i(p^i-p^i_h),w\big)_{\omega_l} \label{lem-low-phi-2-1} \\
 &\le \|\widehat R_{1h}(p^{i}_{h},\phi_{h})\|_{0,\omega_l}\|w\|_{0,\omega_l}
     + C(\|\nabla(\phi-\phi_h)\|_{0,\omega_l}\|w\|_{1,\omega_l} + \sum\limits_{i=1}^n\|p^i-p_h^i\|_{0,\omega_l}\|w\|_{0,\omega_l}).
     \label{lem-low-phi-2}
 \end{align}
 In the following, in order to estimate $\|\widehat J_{h,l}(\phi_{h})\|$, we need introduce an approximation to $\widehat{J}_{h,l}(\phi_{h})$ defined by
 $$\bar{J}_{h,l}(\phi_{h}) = [\bar{\epsilon}\nabla\phi_{h,l}\cdot n_l],$$
 where $\bar{\epsilon}(x)\in S^h$ is a linear interpolation of $\epsilon(x)$ satisfying (cf. \cite{ve94})
 \begin{align}\label{yyeps}
  \|\epsilon(x) -\bar{\epsilon}(x)\|_{0,\infty,l} \le Ch_l\|\epsilon(x)\|_{1,\infty,\omega_l}.
 \end{align}

 For any $w\in L^2(l)$, from (\ref{bouepi}) and \eqref {yyeps}, we have
 \begin{align}
 &\big(\widehat J_{h,l}(\phi_{h})-\bar J_{h,l}(\phi_{h}),w\big)_l
  = \int_l[(\epsilon - \bar\epsilon)\nabla\phi_h\cdot n_l]w \notag\\
              & \leq C \|\epsilon - \bar\epsilon\|_{0,\infty,l}~\|[\nabla\phi_h\cdot n_l]\|_{0,l}\|w\|_{0,l}\notag\\
              & \leq C h_l \|\widehat{J}_{h,l}(\phi_h)\|_{0,l}\|w\|_{0,l}.
 \label{yy-0}
 \end{align}
 Taking $w=\widehat J_{h,l}(\phi_{h})-\bar J_{h,l}(\phi_{h})$ in the above inequality, it yields
 \begin{align}\label{yy-00}
  \|\widehat J_{h,l}(\phi_{h}) - \bar J_{h,l}(\phi_{h})\|_{0,l}
  \le C h_{l}\|\hat{J}_{h,l}(\phi_h)\|_{0,l}.
 \end{align}
 On the other hand, from \eqref{lem-low-phi-2} and \eqref{yy-0}, we get
 \begin{align*}
 \big(\bar J_{h,l}(\phi_{h}),w\big)_l
 &= \big(\bar J_{h,l}(\phi_{h})-\widehat J_{h,l}(\phi_{h}),w\big)_l
     + \big(\widehat J_{h,l}(\phi_{h}),w\big)_l \notag\\
 &\le C \Big(h_l \|\widehat{J}_{h,l}(\phi_h)\|_{0,l}\|w\|_{0,l}+
 \|\widehat R_{1h}(p^{i}_{h},\phi_{h})\|_{0,\omega_l}\|w\|_{0,\omega_l}
      \notag\\
 &+ \|\nabla(\phi-\phi_h)\|_{0,\omega_l}\|w\|_{1,\omega_l} +\sum\limits_{i=1}^n\|p^i-p_h^i\|_{0,\omega_l}\|w\|_{0,\omega_l}\Big).
 \end{align*}
 Similarly, taking $w=\nu_l\bar J_{h,l}(\phi_{h})$  in the above formula and then by Lemma \ref{lem-lower}, we have
 \begin{align*}
  \|\bar J_{h,l}(\phi_h)\|_{0,l}^2
  &\le  C\big(\bar J_{h,l}(\phi_h),\nu_l\bar J_{h,l}(\phi_h)\big)_{l} \\
  & \le C\Big(h^{\frac 1 2}_{\omega_{l}}
    \|\widehat R_{1h}(p^{i}_{h},\phi_{h})\|_{0,\omega_l}
     + h^{-\frac 1 2}_{\omega_{l}}\|\nabla(\phi-\phi_h)\|_{0,\omega_l}
     + h^{\frac 1 2}_{\omega_{l}}\sum\limits_{i=1}^n
       \|p^i-p_h^i\|_{0,\omega_l} \notag\\
  &+ h_l \|\hat{J}_{h,l}(\phi_h)\|_{0,l} \Big)
       \|\bar J_{h,l}(\phi_h)\|_{0,l}.
 \end{align*}
 Hence, by using \eqref{yy-00}, it yields
 \begin{align*}
  &h^{\frac 1 2}_{\omega_{l}}\|\widehat J_{h,l}(\phi_{h})\|_{0,l}
  \le  h_{\omega_l}^{\frac 1 2} \big( \|\bar J_{h,l}(\phi_{h})\|_{0,l} + \|\widehat J_{h,l}(\phi_{h})-\bar J_{h,l}(\phi_{h})\|_{0,l} \big) \notag\\
  &\le C\Big( h_{\omega_l}\|\widehat R_{1h}(p^{i}_{h},\phi_{h})\|_{0,\omega_l}
       + \|\nabla(\phi-\phi_h)\|_{0,\omega_l}
       + h_{\omega_l}\sum\limits_{i=1}^n\|p^i-p_h^i\|_{0,\omega_l} +h_l^{\frac 3 2} \|\widehat{J}_{h,l}(\phi_h)\|_{0,l}\Big).
 \end{align*}
Hence, choosing $h_l$ sufficiently small such that $Ch_l\ll1$, then we obtain
\begin{align}\label{lem-low-phi-3}
  h^{\frac 1 2}_{\omega_{l}}\|\widehat J_{h,l}(\phi_{h})\|_{0,l}
  \le C\Big( h_{\omega_l}\|\widehat R_{1h}(p^{i}_{h},\phi_{h})\|_{0,\omega_l}
       + \|\nabla(\phi-\phi_h)\|_{0,\omega_l}
       + h_{\omega_l}\sum\limits_{i=1}^n\|p^i-p_h^i\|_{0,\omega_l} \Big).
 \end{align}
 Now we turn to estimate $\|\widehat R_{1h}(p^{i}_{h},\phi_{h})\|_{0,\omega_l}$. Define
 \begin{align*}
 \widetilde{M}_h(p^{i}_{h},\phi_{h})|_\tau
 = (\sum\limits_{i=1}^nq^ip_h^i)|_\tau
   + \frac {1}{|\tau|}\int_{\tau}{\rm div}(\epsilon(x)\nabla\phi_h)
   + \frac {1}{|\tau|}\int_{\tau}f.
 \end{align*}
 Since $q^i$, $i=1,2,\cdots,n$ are constants, it is easy to see that $\widetilde{M}_h(p^{i}_{h},\phi_{h})\in \mathcal{P}^1(\tau)$ and
 \begin{align}\label{lem-low-phi-4}
  \|\widehat R_{1h}(p^{i}_{h},\phi_{h}) - \widetilde{M}_h(p^{i}_{h},\phi_{h})\|_{0,\tau}
  \le C h_{\tau}\big( \|\epsilon\|_{2,\infty,\tau}\|\phi_h\|_{1,\tau}+|f|_{1,\tau} \big).
 \end{align}
 In addition, for any $ w\in H^1_{0}(\tau)$, by \eqref{lem-low-phi-2-1}, we obtain
 \begin{align} \label{lem-low-phi-5}
  \big(\widetilde{M}_h(p^{i}_{h},\phi_{h}),w\big)_{\tau}
  = \widetilde{a}(\phi-\phi_{h},w)_{\tau}
    -\big(\widehat R_{1h}(p^{i}_{h},\phi_{h}) - \widetilde{M}_h(p^{i}_{h},\phi_{h}),w\big)_{\tau}
    - \big(\sum_{i=1}^nq^i(p^i-p^i_h),w\big)_{\tau},
 \end{align}
 where we have used $\big(\widehat J_{h,l}(\phi_h),w\big)_l=0$, $l\in \partial\tau$ in \eqref{lem-low-phi-5}, for any $w\in H^1_{0}(\tau)$.
 Taking $w=\mu_{\tau}\widetilde{M}_h(p^{i}_{h},\phi_{h})$
 in \eqref{lem-low-phi-5} and by Lemma \ref{lem-lower}, there holds
 \begin{align*}
  &\|\widetilde{M}_h(p^{i}_{h},\phi_{h})\|_{0,\tau}^2
   \le C\big(\widetilde{M}_h(p^{i}_{h},\phi_{h}),
           \mu_{\tau}\widetilde{M}_h(p^{i}_{h},\phi_{h})\big)_\tau \notag\\
  &\le C\Big( \|\nabla(\phi-\phi_h)\|_{0,\tau}\|\mu_{\tau}
            \widetilde{M}_h(p^{i}_{h},\phi_{h})\|_{1,\tau}
        + \|\widehat R_{1h}(p^{i}_{h},\phi_{h})
            - \widetilde{M}_h(p_h^i,\phi_h)\|_{0,\tau}
         \|\mu_{\tau}\widetilde{M}_h(p^{i}_{h},\phi_{h})\|_{0,\tau} \notag\\
  &\quad +\sum_{i=1}^n\|p^i-p^i_h\|_{0,\tau}
          \|\mu_{\tau}\widetilde{M}_h(p^{i}_{h},\phi_{h})\|_{0,\tau} \Big) \notag\\
  &\le C \Big( h_\tau^{-1}\|\nabla(\phi-\phi_h)\|_{0,\tau}
         + \|\widehat R_{1h}(p^{i}_{h},\phi_{h})
            - \widetilde{M}_h(p_h^i,\phi_h)\|_{0,\tau}
         + \sum_{i=1}^n\|p^i-p^i_h\|_{0,\tau} \Big)
            \|\widetilde{M}_h(p^{i}_{h},\phi_{h})\|_{0,\tau}.
 \end{align*}
 Thus, we get
 \begin{align}\label{lem-low-phi-6-0}
  \|\widetilde{M}_h(p^{i}_{h},\phi_{h})\|_{0,\tau}
  \le C \Big(  h_\tau^{-1} \|\nabla(\phi-\phi_h)\|_{0,\tau}
       +\|\widehat R_{1h}(p^{i}_{h},\phi_{h})-\widetilde{M}_h(\phi_h)\|_{0,\tau}
       + \sum_{i=1}^n\|p^i-p^i_h\|_{0,\tau}\Big).
 \end{align}
 Combining \eqref{lem-low-phi-4} and \eqref{lem-low-phi-6-0}, it yields
 \begin{align}\label{lem-low-phi-7}
  \|\widehat R_{1h}(p^{i}_{h},\phi_{h})\|_{0,\tau}
  &\le\|\widetilde{M}_h(p_h^i,\phi_h)\|_{0,\tau}
    +\|\widehat R_{1h}(p^{i}_{h},\phi_{h})-\widetilde{M}_h(p_h^i,\phi_h)\|_{0,\tau} \notag\\
  & \le C\Big( h_\tau^{-1}\|\nabla(\phi-\phi_h)\|_{0,\tau}
    + h_{\tau}\big( \|\epsilon\|_{2,\infty,\tau}\|\phi_h\|_{1,\tau} + |f|_{1,\tau} \big)
    + \sum_{i=1}^n\|p^i-p^i_h\|_{0,\tau} \Big).
 \end{align}
 Substituting \eqref{lem-low-phi-7} into \eqref{lem-low-phi-3}, we obtain
 \begin{align*}
  h^{\frac 1 2}_{\omega_{l}}\|\widehat J_{h,l}(\phi_{h})\|_{0,l}
  \le C\Big( \|\nabla(\phi-\phi_h)\|_{0,\omega_l}
    + h_{\omega_l}\sum\limits_{i=1}^n\|p^i-p_h^i\|_{0,\omega_l} \Big)
    + r_{\phi_h,\omega_l}.
 \end{align*}
 where $ r_{\phi_h,\omega_l}
 \le Ch^2_{\omega_l} \big( |f|_{1,\omega_l}+ \|\phi_h\|_{1,\omega_l}\|\epsilon\|_{2,\infty,\omega_l}\big)$.

 This completes the proof of Lemma \ref{lem-low-phi}. $\hfill\Box$
\end{proof}

 Applying the above results, we have the following lower bound for $\|\nabla(\phi-\phi_h)\|_{0,\omega_\tau}$.
 \begin{theorem}\label{th-phi-lower}
   Let $(p^i, \phi)$ and $(p_h^i, \phi_h)$ be the solutions of \eqref{weak-p}-\eqref{weak-phi} and \eqref{fem-ph}-\eqref{fem-phih}, respectively. For any $ \tau\in \mathcal{T}^h$, there holds
 \begin{align} \label{th-phi-lower-0}	
   \eta_{\tau,\phi}(p^i_{h},\phi_{h})
   \le C \xi_{h, \omega_{\tau}}
     \Big( \|\nabla(\phi-\phi_h)\|_{0,\omega_\tau}
      + h_{\omega_\tau}\sum\limits_{i=1}^n\|p^i-p_h^i\|_{0,\omega_\tau} \Big)
      + \widetilde{r}_{\phi_h,\omega_\tau} ,
 \end{align}
 where
 \begin{align*}
  &\eta_{\tau,\phi}(p^i_{h},\phi_{h})
    =  h_\tau\|R_{1h}(p^i_{h},\phi_{h})\|_{0,\tau}
           + \|D_{h}(\phi_{h})\|_{0,\tau}, \\
  & R_{1h}(p_h^i,\phi_h)=\sum_{i=1}^nq^ip_h^i+{\rm div}(G_h\phi_h)+f,
    ~~D_h(\phi_h)= G_{h}\phi_{h} - \epsilon(x)\nabla\phi_{h}, \\
  &\widetilde{r}_{\phi_h,\omega_\tau}
   \le C \xi_{h, \omega_{\tau}} h_{\omega_{\tau}}^{2}
            \big( \|\epsilon\|_{2,\infty,\omega_\tau}\|\phi_h\|_{1,\omega_\tau} +|f|_{1,\omega_\tau} \big), \\
  &\xi_{h, \omega_{\tau}}=\max _{l \in \partial \mathcal{T}^{h}, l \subset \omega_{\tau} \backslash \partial \omega_{\tau}} \xi_{h, l}, ~~
   \xi_{h, l}=1+h_{l}|\epsilon|_{1, \infty, l}\left\|\epsilon^{-1}\right\|_{0, \infty, l}.
 \end{align*}
\end{theorem}

\begin{proof}
 From the definition of $G_h$, we get
 $$ G_h\phi_h - \epsilon(x)\nabla\phi_h = \sum_{z \in \partial^{2} \mathcal{T}^{h}}(\epsilon(x)\nabla\phi_h)_{z} \varphi_{z}-\epsilon(x)\nabla\phi_h, $$
 where $\varphi$ is the basis function. Thus $\forall x\in\tau, \tau\in\mathcal{T}^h$, if $\{z_i:i=1,2,3\}$ is the vertex set of $\tau$, then by using the similar arguments as (3.19)-(3.22) in \cite{Y.Yang2006}, we have
 \begin{align} \label{lower-Ghphi0}
  G_{h} \phi_h - \epsilon(x) \nabla \phi_h
  &=\sum_{i=1}^{3} \varphi_{z_i}(x)\left(\sum_{j=1}^{J_{z_i}}
    \alpha_{z_{i}}^{j}\big(\epsilon\left(z_{i}\right) \nabla \phi_h\big)_{\tau_{z_i}^{j}}\right)-(\epsilon \nabla \phi_h)(x) \notag\\
  &= \sum_{i=1}^{3} \varphi_{z_{i}}(x) \sum_{j=1}^{J_{z_i}}
     \alpha_{z_{i}}^{j}\left(\big(\epsilon\left(z_{i}\right) \nabla \phi_h\big)_{\tau_{z_i}^{j}}-\big(\epsilon(x) \nabla \phi_h\big)_{\tau}\right) \notag\\
  &= \sum_{i=1}^{3} \varphi_{z_{i}}(x) \sum_{j=1}^{J_{z_i}}
     \alpha_{z_{i}}^{j}\left(\big(\epsilon\left(z_{i}\right) \nabla \phi_h\big)_{\tau_{z_i}^{j}}-\big(\epsilon\left(z_{i}\right) \nabla \phi_h\big)_{\tau}\right) \notag\\
  &\quad + \left(\sum_{i=1}^{3} \varphi_{z_{i}}(x)
     \big(\epsilon\left(z_{i}\right)\big)_{\tau} - \epsilon(x)\right)(\nabla \phi_h)_{\tau}.
 \end{align}
 We can find a cluster of simplices $\tau', \tau_{1}, \cdots, \tau_{K}, \tau \in \omega_{z}$, such that $\overline{\tau}_{k} \cap \overline{\tau}_{k+1}= l_{k} \in \partial \mathcal{T}^{h}~(k=0,1, \cdots, K+1)$, where $\tau_0=\tau'$ and $\tau_{K+1}=\tau$. Thus
\begin{align*}
 \big(\epsilon\left(z_{i}\right) \nabla \phi_h\big)_{\tau^{\prime}}-\big(\epsilon\left(z_{i}\right) \nabla \phi_h\big)_{\tau}
 &=\sum_{k=0}^{K}\left(\big(\epsilon\left(z_{i}\right) \nabla\phi_h\big)_{\tau_{k}}
   -\big(\epsilon\left(z_{i}\right) \nabla \phi_h\big)_{\tau_{k+1}}\right) \\
 &=\sum_{k=0}^{K}\left(\big(\epsilon\left(z_{i}\right) \nabla \phi_h\big)_{\tau_{k}}\cdot
   n_{l_{k}}-\big(\epsilon\left(z_{i}\right) \nabla\phi_h\big)_{\tau_{k+1}} \cdot n_{l_{k}}\right) n_{l_k}.
\end{align*}
 That is
 \begin{equation} \label{lower-phi00}
 \big(\epsilon\left(z_{i}\right) \nabla \phi_h\big)_{\tau^{\prime}}-\big(\epsilon\left(z_{i}\right) \nabla \phi_h\big)_{\tau}=\sum_{k=0}^{K}\left[\epsilon\left(z_{i}\right) \nabla \phi_h \cdot n_{l_{k}}\right] n_{l_{k}},
 \end{equation}
 where $n_{l_k}$ is the unit normal vector to $l_k$. If $z$ is a vertex of $l_k$, then
 \begin{align} \label{lower-phi01}
 \Big\|\left[\epsilon(z) \nabla \phi_h \cdot n_{l_{k}}\right]\Big\|_{0, l_{k}}
 &\le \Big\|[\epsilon \nabla \phi_h \cdot n_{l_{k}}]\Big\|_{0,l_{k}}
     +\Big\|\left[\big(\epsilon - \epsilon(z)\big) \epsilon^{-1} \epsilon \nabla \phi_h \cdot n_{l_{k}}\right]\Big\|_{0, l_{k}}  \notag\\
 &\le C \xi_{h, l_k}
      \Big\|\left[\epsilon \nabla \phi_h \cdot n_{l_{k}}\right]\Big\|_{0, l_{k}},
 \end{align}
 where
 $\xi_{h, l_k}=1+h_{l_k}|\epsilon|_{1, \infty, l_k}\left\|\epsilon^{-1}\right\|_{0, \infty, l_k}$.

 Obviously, we know that $\sum\limits_{i=1}^{3} \varphi_{z_{i}}(x)\big(\epsilon\left(z_{i}\right)\big)_{\tau}$ is the Lagrange interpolation of $\epsilon(x)$ and
 \begin{equation} \label{lower-phi02}
 \left\|\left(\sum_{i=1}^{3}
  \varphi_{z_{i}}(x)\big(\epsilon\left(z_{i}\right)\big)_{\tau}-\epsilon(x)\right)
   \nabla\phi_h\right\|_{0, \tau} \le C h_{\tau}^{2}|\epsilon|_{2, \infty,\tau}|\phi_h|_{1,\tau}.
 \end{equation}
 Then from \eqref{lem-low-phi-0} and \eqref{lower-Ghphi0}-\eqref{lower-phi02}, we get
 \begin{align} \label{lower-phi03}
  \left\|D_{h}(\phi_h)\right\|_{0, \tau}
    &\le C  \Big(\sum_{i=1}^{3} \sum_{l \in \partial \mathcal{T}^{h}, l \subset
      \omega_{z_{i}}} h_l^{\frac 1 2}\xi_{h,l}\|[\epsilon\nabla\phi_h\cdot n_l]\|_{0,l}+h_{\tau}^{2}|\epsilon|_{2,\infty,\tau}|\phi_h|_{1,\tau}\Big)\notag\\
 &\le C \xi_{h, \omega_{\tau}}
     \Big( \|\nabla(\phi-\phi_h)\|_{0,\omega_\tau}
      + h_{\omega_\tau}\sum\limits_{i=1}^n\|p^i-p_h^i\|_{0,\omega_\tau}
      + r_{\phi_h,\omega_\tau} \Big)
      + C h_{\tau}^{2}|\epsilon|_{2,\infty, \tau}|\phi_h|_{1,\tau} \notag\\
 &\le C \xi_{h, \omega_{\tau}}
     \Big( \|\nabla(\phi-\phi_h)\|_{0,\omega_\tau}
      + h_{\omega_\tau}\sum\limits_{i=1}^n\|p^i-p_h^i\|_{0,\omega_\tau} \Big)
      + \widetilde{r}_{\phi_h,\omega_\tau},
 \end{align}
 where
 \begin{align*}
  &\widetilde{r}_{\phi_h,\omega_\tau}
   \le C \xi_{h, \omega_{\tau}} h_{\omega_{\tau}}^{2}
            \big( \|\epsilon\|_{2,\infty,\omega_\tau}\|\phi_h\|_{1,\omega_{\tau}} +|f|_{1,\omega_\tau} \big), \\
 &\xi_{h, \omega_{\tau}}=\max _{l \in \partial \mathcal{T}^{h}, l \subset \omega_{\tau} \backslash \partial \omega_{\tau}} \xi_{h, l}, ~~ \xi_{h, l}=1+h_{l}|\epsilon|_{1, \infty, l}\left\|\epsilon^{-1}\right\|_{0, \infty, l}.
 \end{align*}

 Next, we only need to estimate $ \|R_{1h}(p_h^i,\phi_{h})\|_{0,\tau}$. Define
 \begin{align*}
  \widetilde{R}_{1h}(p^i_{h},\phi_{h})\big|_\tau
   =\big({\rm div}(G_{h}\phi_{h})\big)|_{\tau}+\sum_{i=1}^nq^ip^i_{h}+\frac {1} {|\tau|}\int_{\tau}f.
 \end{align*}
 It is seen that $ \widetilde{R}_{1h}(p^i_{h},\phi_{h})\in \mathcal{P}^1(\tau)$ and
 \begin{align}\label{th-phi-lower-1}
  \|R_{1h}(p^i_{h},\phi_{h})-\widetilde{R}_{1h}(p^i_{h},\phi_{h})\|_{0,\tau}
  \le C h_{\tau} |f|_{1,\tau}.
 \end{align}
 On the other hand, for any $ w\in
 H^1_{0}(\Omega)$, taking $\chi=0$ in \eqref{upphi-3-0}, it yields
 \begin{align} \label{th-phi-lower-2}
 \widetilde{a}(\phi-\phi_h,w)
 =\big(R_{1h}(p_h^i,\phi_h),w\big)+\big(\sum_{i=1}^nq^i(p^i-p^i_h),w\big)
   +\big(D_h(\phi_h),\nabla w\big),
 \end{align}
 where
 $R_{1h}(p_h^i,\phi_h)=\sum\limits_{i=1}^nq^ip_h^i + {\rm div}(G_h\phi_h) + f$ and $D_h(\phi_h)= G_{h}\phi_{h} - \epsilon(x)\nabla\phi_{h}$.

 Further, from \eqref{th-phi-lower-2}, we have
 \begin{align} \label{th-phi-lower-3}
  \big(\widetilde{R}_{1h}(p_h^i,\phi_h),w\big)_\tau
  &=\widetilde{a}(\phi-\phi_h,w)_\tau
   -\big(R_{1h}(p^i_{h},\phi_{h})
   -\widetilde{R}_{1h}(p^i_{h},\phi_{h}),w\big)_\tau \notag \\
  &\quad -\big(\sum_{i=1}^nq^i(p^i-p^i_h),w\big)_\tau
         -\big(D_h(\phi_h),\nabla w\big)_\tau.
  \end{align}
 Taking $w=\mu_\tau\widetilde{R}_{1h}(p_h^i,\phi_h)$ in \eqref{th-phi-lower-3}, by Lemma \ref{lem-lower}, we get
 \begin{align*} 
  \|\widetilde{R}_{1h}(p_h^i,\phi_h)\|_{0,\tau}^2
  &\le C \big(\widetilde{R}_{1h}(p_h^i,\phi_h),
             \mu_\tau\widetilde{R}_{1h}(p_h^i,\phi_h)\big) \notag\\
  &\le C\Big( h_\tau^{-1} \|\nabla(\phi-\phi_h)\|_{0,\tau}
       +\|R_{1h}(p^i_{h},\phi_{h})-\widetilde{R}_{1h}(p^i_{h},\phi_{h})\|_{0,\tau}  \notag \\
  &\quad + \sum_{i=1}^n\|p^i-p^i_h\|_{0,\tau}
         +  h_\tau^{-1}\|D_h(\phi_h)\|_{0,\tau}\Big)
            \|\widetilde{R}_{1h}(p^i_{h},\phi_{h})\|_{0,\tau}.
 \end{align*}
 This implies that
 \begin{align} \label{th-phi-lower-4-0}
  \|\widetilde{R}_{1h}(p_h^i,\phi_h)\|_{0,\tau}
   &\le C\Big(  h_\tau^{-1} \|\nabla(\phi-\phi_h)\|_{0,\tau}
       +\|R_{1h}(p^i_{h},\phi_{h})-\widetilde{R}_{1h}(p^i_{h},\phi_{h})\|_{0,\tau} \notag\\
  &\quad + \sum_{i=1}^n\|p^i-p^i_h\|_{0,\tau}
         + h_\tau^{-1}\|D_h(\phi_h)\|_{0,\tau} \Big).
 \end{align}
 Then, by \eqref{th-phi-lower-1} and \eqref{th-phi-lower-4-0}, we get
 \begin{align} \label{th-phi-lower-4-1}
 &\|R_{1h}(p_h^i,\phi_h)\|_{0,\tau}
   \le \|\widetilde{R}_{1h}(p_h^i,\phi_h)\|_{0,\tau}
      + \|R_{1h}(p^i_{h},\phi_{h})-\widetilde{R}_{1h}(p^i_{h},\phi_{h})\|_{0,\tau}
    \notag \\
 &\le C \Big( h_\tau^{-1} \|\nabla(\phi-\phi_h)\|_{0,\tau}
       +  h_{\tau} |f|_{1,\tau} + \sum_{i=1}^n\|p^i-p^i_h\|_{0,\tau}
       +  h_\tau^{-1}\|D_h(\phi_h)\|_{0,\tau} \Big).
 \end{align}
 Hence, from \eqref{lower-phi03} and \eqref{th-phi-lower-4-1}, there holds
 \begin{align} \label{th-phi-lower-4-2}
  h_\tau\|R_{1h}(p_h^i,\phi_h)\|_{0,\tau}
  \le C\xi_{h, \omega_{\tau}}
     \Big( \|\nabla(\phi-\phi_h)\|_{0,\omega_\tau}
  + h_{\omega_\tau}\sum\limits_{i=1}^n
       \|p^i-p_h^i\|_{0,\omega_\tau} \Big)
  + \widetilde{r}_{\phi_h,\omega_\tau}.
 \end{align}
 where
 \begin{align*}
  &\widetilde{r}_{\phi_h,\omega_\tau}
   \le C \xi_{h, \omega_{\tau}} h_{\omega_{\tau}}^{2}
            \big(  \|\epsilon\|_{2,\infty,\omega_\tau}\|\phi_h\|_{1,\omega_{\tau}} +|f|_{1,\omega_\tau} \big), \\
 &\xi_{h, \omega_{\tau}}=\max _{l \in \partial \mathcal{T}^{h}, l \subset \omega_{\tau} \backslash \partial \omega_{\tau}} \xi_{h, l}, ~~ \xi_{h, l}=1+h_{l}|\epsilon|_{1, \infty, l}\left\|\epsilon^{-1}\right\|_{0, \infty, l}.
 \end{align*}

 Then the desired result \eqref{th-phi-lower-0} can be obtained by \eqref{lower-phi03} and \eqref{th-phi-lower-4-2}. This completes the proof.   $\hfill\Box$
\end{proof}

Now we turn to derive the lower bound of the a posteriori error indicator for $\|\nabla(p^i-p_h^i)\|_{0,\omega_\tau}$. First, we need the following lemmas.
\\

The following result was shown in \cite{N.Yan2001} for the gradient recovery operator $\widetilde{G}_h$ defined in (\ref{Ghwh0-0}).
\begin{lemma}\cite{N.Yan2001} \label{xlemm}
    For any $v_{h}\in S_{0}^h$, $ \tau\in \mathcal{T}^h$, $l\in \partial \mathcal{T}^h,l\not\subset\partial\Omega$, there holds
 \begin{equation} \label{ghrec-vh}
  \|\widetilde{G}_{h}v_{h} - \nabla v_{h}\|_{0,\tau}
  \le C\sum_{l\subset (\omega_{\tau}\setminus \partial\omega_{\tau})}h^{\frac{1} {2}}_{{l}}\|\widetilde{J}_{h,l}(v_{h})\|_{0,l},
 \end{equation}
 where $\widetilde{J}_{h,l}(v_{h}) = [\nabla v_h\cdot n_l]$.
\end{lemma}

 Furthermore, from (\ref{bouepi}) and \eqref{epsilon-jh}, we can easily get
 \begin{align} \label{jh-new}
 \|\widetilde{J}_{h,l}(v_{h})\|_{0,l}
 \le C\|\widehat{J}_{h,l}(v_{h})\|_{0,l}.
 \end{align}
 Then for $\phi_{h}\in S_{0}^h$, $ \tau\in \mathcal{T}^h$, $l\in \partial \mathcal{T}^h,l\not\subset\partial\Omega$, from \eqref{lem-low-phi-0}, \eqref{ghrec-vh} and \eqref{jh-new}, we have
 \begin{equation}\label{jh-new-1}
  \|\widetilde{G}_{h}\phi_{h} - \nabla \phi_{h}\|_{0,\tau}
  \le C\Big( \|\nabla(\phi-\phi_h)\|_{0,\omega_\tau} + h_{\omega_\tau}\sum\limits_{i=1}^n\|p^i-p_h^i\|_{0,\omega_\tau}\Big)
       +r_{\phi_h,\omega_\tau},
 \end{equation}
where $r_{\phi_h,\omega_\tau}\leq C h_{\omega_\tau}^2(\|\epsilon\|_{2,\infty,\omega_\tau}\|\phi_h\|_{1,\omega_\tau}+\|f\|_{1,\omega_\tau})$.

 By using the above results, similar to Lemma \ref{lem-low-phi}, we have the following lemma.
\begin{lemma}\label{lem-low-jhp0}
  Let $(p^i, \phi)$ and $(p_h^i, \phi_h)$ be the solutions of \eqref{weak-p}-\eqref{weak-phi} and \eqref{fem-ph}-\eqref{fem-phih}, respectively. For any $l\in \partial \mathcal{T}^h$, $l\not\subset\partial\Omega$, if $h_l\ll 1$, the coefficients $\alpha,~\beta,~g$ and $\gamma$ of (\ref{MPNP-model}) satisfy
   \begin{align} \label{lipcond}
  \|\alpha(x,p^i)-\alpha(x,p^i_h)\|_{0,\omega_l}&+\|\beta(x,p^i)-\beta(x,p^i_h)\|_{0,\omega_l}
  +\|g(x,p^i)-g(x,p^i_h)\|_{0,\omega_l}+\|\gamma(x,p^i)-\gamma(x,p^i_h)\|_{0,\omega_l}\notag\\
  &\leq C \|p^i-p_h^i\|_{0,\omega_l}
 \end{align}
 and $f\in H^1(\ome)$, then there holds
 \begin{equation}\label{lem-low-jhp0-0}
  h^{\frac 1 2}_{\omega_l}\|J_{h,l}(p_h^i)\|_{0,l}
  \le C \big( \|\nabla(p^i-p_h^i)\|_{0,\omega_\tau}
   + \|\nabla(\phi-\phi_h)\|_{0,\omega_\tau}
   + \sum\limits_{i=1}^n\|p^i-p_h^i\|_{0,\omega_\tau} \big)
   + r_{p_h^i,\omega_\tau},
 \end{equation}
 where
 \beas J_{h,l}(p_h^i)=[\alpha(x,p_h^i)\nabla p_h^i \cdot n_l]
 \eeas
 and
 $$r_{p_h^i,\omega_\tau} \le C h_{\omega_\tau}^2 \big( \|\alpha(x,p_h^i)\|_{2,\infty,\omega_\tau}\|p_h^i\|_{1,\omega_\tau}
        + |\beta(x,p_h^i)|_{2,\omega_\tau} + |g(x,p_h^i)|_{1,\omega_\tau}+\|\epsilon\|_{2,\infty,\omega_\tau}\|\phi_h\|_{1,\omega_\tau}
        + |f|_{1,\omega_\tau}\big).$$
\end{lemma}

\begin{proof}
 First, for any $v\in H_{0}^1(\Omega)$, from \eqref{aform1} and Green's formula, we have
 \begin{align} \label{a-ph-Jh0-0}
 a(p_h^i,v)
  &=-\sum_{\tau\in\mathcal{T}^h}\int_{\tau}
   \Big({\rm div}\big(\alpha(x,p_h^i)\nabla p_h^i\big)
   +{\rm div}\big(\beta(x,p_h^i)\big)-g(x,p_h^i)\Big)v
   -\sum_{l\in\partial\mathcal{T}^h,l\not\subset \partial\Omega}\int_l[\alpha(x,p_h^i)\nabla p_h^i\cdot n_l]v
    \notag \\
 &=-\sum_{\tau\in\mathcal{T}^h}\int_{\tau} R^h(p_h^i) v
   +J_h(p_h^i,v),
 \end{align}
 where
 \beas R^h(p_h^i)|_{\tau} = \Big({\rm div}\big(\alpha(x,p_h^i)\nabla p_h^i\big)+ {\rm div}\big(\beta(x,p_h^i)\big)-g(x,p_h^i)\Big)|_\tau.
 \eeas
 and
 \beas J_{h}(p_h^i,v)=\sum_{l\in\partial\mathcal{T}^h,l\not\subset \partial\Omega}\int_l[\alpha(x,p_h^i)\nabla p_h^i\cdot n_l]v
     \eeas
Hence from (\ref{weak-p}) and \eqref{a-ph-Jh0-0}, there holds
 \begin{align} \label{a-ph-Jh0}
  J_{h}(p_h^i,v)
  &= a(p_h^i,v)+\sum_{\tau\in \mathcal{T}^h}(R^h(p_h^i),v)_\tau\notag\\
  &=a(p_h^i,v)-a(p^i,v)-b(p^i,\phi,v)+\sum_{\tau\in \mathcal{T}^h}(R^h(p_h^i),v)_\tau\notag\\
  &=\Big((a(p_h^i,v)-a(p^i,v)\Big)+\Big(b(p_h^i,\phi_h,v)-b(p^i,\phi,v)\Big)-b(p_h^i,\phi_h,v)+\sum_{\tau\in \mathcal{T}^h}(R^h(p_h^i),v)_\tau
 \end{align}

 From (\ref{b-a2form}) and using the gradient recovery operator $\tilde{G}_h$, we rewrite the third term on the right-hand side of \eqref{a-ph-Jh0} as follows:
 \begin{align} \label{gamma-phphih}
  -b(p_h^i,\phi_h,v)&=-\big(  \gamma(x,p_h^i)\nabla\phi_h,\nabla v\big)\notag\\
   &= \big(\gamma(x,p_h^i)(\widetilde{G}_h\phi_h-\nabla\phi_h),
        \nabla v\big) -\big(\gamma(x,p_h^i)\widetilde{G}_h\phi_h),
           \nabla v\big) \notag\\
  &=\big(\gamma(x,p_h^i)(\widetilde{G}_h\phi_h-\nabla\phi_h),
         \nabla v\big)
    + \Big({\rm div}
      \big(\gamma(x,p_h^i)\widetilde{G}_h\phi_h\big),v\Big) \notag\\
  &=  \Big( \big(\gamma(x,p_h^i)-\gamma(x,p^i)\big)
     (\widetilde{G}_h\phi_h-\nabla\phi_h),\nabla v \Big)
      +\big(\gamma(x,p^i)(\widetilde{G}_h\phi_h-\nabla\phi_h),\nabla v \big) \notag\\
  &\quad + \Big({\rm div}
   \big(\gamma(x,p_h^i)\widetilde{G}_h\phi_h\big),v\Big).
 \end{align}
 Substituting \eqref{gamma-phphih} into \eqref{a-ph-Jh0}, we have
 \begin{align}
 J_{h}&(p_h^i,v)
  = \Big((a(p_h^i,v)-a(p^i,v)\Big)+\Big(b(p_h^i,\phi_h,v)-b(p^i,\phi,v)\Big) + \big(\widehat{R}^h(p_h^i,\phi_h),v\big)
  \notag\\
  &\quad  + \Big( \big(\gamma(x,p_h^i)-\gamma(x,p^i)\big)
     (\widetilde{G}_h\phi_h-\nabla\phi_h),\nabla v \Big)
    + \big(\gamma(x,p^i)(\widetilde{G}_h\phi_h-\nabla\phi_h), \nabla v\big),  \label{a-ph-Jh-00}
 \end{align}
 where
 $$\widehat{R}^h(p_h^i,\phi_h) = {\rm div}\big(\alpha(x,p_h^i)\nabla p_h^i\big) + {\rm div}\big(\beta(x,p_h^i)\big)-g(x,p_h^i) + {\rm div}\big(\gamma(x,p_h^i)\widetilde{G}_h\phi_h\big).$$
 By using (\ref{lipcond}), for $v\in H_0^1(\omega_l)$, it is easy to get
 \bea \label{bound} (a(p_h^i,v)-a(p^i,v)+b(p_h^i,\phi_h,v)-b(p^i,\phi,v)\leq C ( \|p^i-p_h^i\|_{1,\omega_l}+\|\nabla(\phi-\phi_h)\|_{0,\omega_l})\|v\|_{1,\omega_l}.
 \eea
 Then for any $v\in H_0^1({\omega_l})$, inserting \eqref{bound} into \eqref{a-ph-Jh-00} and from \eqref{Ghwh-infty}, \eqref{phih-infty-new0} and (\ref{lipcond}),  we get
 \begin{align}
 (J_{h,l}&(p_h^i),v)_l=J_{h}(p_h^i,v)\notag\\
  &\le C\Big( \|p^i-p_h^i\|_{1,\omega_l} \|v\|_{1,\omega_l}
    + \|\gamma(x,p_h^i)-\gamma(x,p^i)\|_{0,\omega_l}
      \|\widetilde{G}_h\phi_h-\nabla\phi_h\|_{0,\infty,\omega_l}
      \|v\|_{1,\omega_l} \notag \\
  &\quad +\|\widehat{R}^h(p_h^i,\phi_h)\|_{0,\omega_l}\|v\|_{0,\omega_l}
     + \|\gamma(x,p^i)\|_{0,\infty,\omega_l}
       \|\widetilde{G}_h\phi_h-\nabla\phi_h\|_{0,\omega_l}
        \|v\|_{1,\omega_l} \Big) \notag\\
  &\le C\Big( \|p^i-p_h^i\|_{1,\omega_l}\|v\|_{1,\omega_l}
   +\|\widehat{R}^h(p_h^i,\phi_h)\|_{0,\omega_l}\|v\|_{0,\omega_l}
   + \|\widetilde{G}_h\phi_h-\nabla\phi_h\|_{0,\omega_l}
        \|v\|_{1,\omega_l}  \notag\\
  &\quad + \|\nabla(\phi-\phi_h)\|_{0,\omega_l}
           \|v\|_{1,\omega_l} \Big)
   \label{a-ph-Jh-1}.
 \end{align}

 In order to estimate $\| J_{h,l}(p_h^i)\|_{0,l}$, similar to the estimation for $\|\widehat J_{h,l}(\phi_{h})\|_{0,l}$ in Lemma \ref{lem-low-phi}, , we introduce an approximation to $J_{h,l}(p_h^i)$, which is defined as follows
 $$ \bar J_{h,l}(p_h^i) = [\bar\alpha(x,p_h^i)\nabla p_h^i \cdot n_l],$$
 where $\bar\alpha(x,p_h^i)\in S^h$ is a linear interpolation of $\alpha(x,p_h^i)$ on $\omega_l$ satisfying (cf. \cite{ve94})
\begin{align}\label{yyalpha}
 \|\alpha(x,p_h^i) - \bar\alpha(x,p_h^i)\|_{0,\infty,l}
  \le Ch_l \|\alpha(x,p_h^i)\|_{2,\infty,\omega_l}.
 \end{align}
 For any $w\in H_0^1(\omega_l)$, by using (\ref{boualp}) and \eqref{yyalpha}, we get
 \begin{align}
 &\big(J_{h,l}(p_h^i)-\bar J_{h,l}(p_h^i),w\big)_l
  = \int_l[(\alpha(x,p_h^i) - \bar\alpha(x,p_h^i))\nabla p_h^i\cdot n_l]w \notag\\
 &\le C \|\alpha(x,p_h^i) - \bar\alpha(x,p_h^i)\|_{0,\infty,l}
        \|[\nabla p_h^i\cdot n_l]\|_{0,l}\|w\|_{0,l} \notag\\
 &\le C h_l\|\alpha(x,p_h^i)\|_{1,\infty,\omega_l}
           \|[\nabla p_h^i\cdot n_l]\|_{0,l}\|w\|_{0,l} \notag\\
 &\le C h_l\|J_{h,l}(p_h^i)\|_{0,l}\|w\|_{0,l}.
        \label{yyal-0}
 \end{align}
 Taking $w=J_{h,l}(p_h^i)-\bar J_{h,l}(p_h^i)$ in the above formula, it yields
 \begin{align}\label{yyal-1}
  \|J_{h,l}(p_h^i)-\bar J_{h,l}(p_h^i)\|_{0,l}
  \le C h_l\|J_{h,l}(p_h^i)\|_{0,l}.
 \end{align}
 On the other hand, from \eqref{a-ph-Jh-1} and \eqref{yyal-0}, for $w\in H_0^1(\omega_l)$, we have
 \begin{align*}
 \big(\bar J_{h,l}(p_h^i),w\big)_l
 &= \big(\bar J_{h,l}(p_h^i)-J_{h,l}(p_h^i),w\big)_l
     + \big(J_{h,l}(p_h^i),w\big)_l \notag\\
 &\le C \Big(  h_l\|J_{h,l}(p_h^i)\|_{0,l}\|w\|_{0,l}
   + \|p^i-p_h^i\|_{1,\omega_l}\|w\|_{1,\omega_l}
   + \|\widehat{R}^h(p_h^i,\phi_h)\|_{0,\omega_l}\|w\|_{0,\omega_l} \notag\\
 &\quad + \|\widetilde{G}_h\phi_h-\nabla\phi_h\|_{0,\omega_l}
           \|w\|_{1,\omega_l}
   + \|\nabla(\phi-\phi_h)\|_{0,\omega_l}\|w\|_{1,\omega_l}  \Big).
 \end{align*}
 Taking $w=\nu_l\bar J_{h,l}(p_h^i)$ in the above formula and by \eqref{lem-lower-3}-\eqref{lem-lower-5}, there holds
 \begin{align*}
  \|\bar J_{h,l}(p_h^i)\|_{0,l}^2
  &\le  C\big(\bar J_{h,l}(p_h^i),\nu_l\bar J_{h,l}(p_h^i)\big)_{l} \\
  & \le C\Big(
        h_{\omega_l}^{-\frac{1}{2}}\|p^i-p_h^i\|_{1,\omega_l}
      + h_{\omega_l}^{\frac{1}{2}}\|\widehat{R}^h(p_h^i,\phi_h)\|_{0,\omega_l} \notag\\
  &\quad  + h_{\omega_l}^{-\frac{1}{2}} \|\widetilde{G}_h\phi_h-\nabla\phi_h\|_{0,\omega_l}
  + h_{\omega_l}^{-\frac{1}{2}}\|\nabla(\phi-\phi_h)\|_{0,\omega_l}
  + h_l\|J_{h,l}(p_h^i)\|_{0,l} \Big)
           \|\bar{J}_{h,l}(p_h^i)\|_{0,l}.
 \end{align*}
From the above inequality and by using \eqref{yyal-1}, we get
\begin{align*}
  h_{\omega_l}^{\frac{1}{2}}\|J_{h,l}(p_h^i)\|_{0,l}
  &\le  h_{\omega_l}^{\frac{1}{2}} \big( \|\bar J_{h,l}(p_h^i)\|_{0,l}
  + \|J_{h,l}(p_h^i)-\bar J_{h,l}(p_h^i)\|_{0,l} \big) \notag\\
  &\le C \Big( \|p^i-p_h^i\|_{1,\omega_l}
         + h_{\omega_l}\|\widehat{R}^h(p_h^i,\phi_h)\|_{0,\omega_l}+ \|\widetilde{G}_h\phi_h-\nabla\phi_h\|_{0,\omega_l}  \notag\\
  &\quad + \|\nabla(\phi-\phi_h)\|_{0,\omega_l}
   + h_l^{\frac{3}{2}}\|J_{h,l}(p_h^i)\|_{0,l}
         \Big).
 \end{align*}
Hence, choosing $h_l$ sufficiently small such that $C h_l\ll 1$, we obtain
\begin{align} \label{lower-jh-1}
  h_{\omega_l}^{\frac{1}{2}}\|J_{h,l}(p_h^i)\|_{0,l}
  &\le C \Big( \|p^i-p_h^i\|_{1,\omega_l}
         + h_{\omega_l}\|\widehat{R}^h(p_h^i,\phi_h)\|_{0,\omega_l}+ \|\widetilde{G}_h\phi_h-\nabla\phi_h\|_{0,\omega_l}  \notag\\
  &\quad + \|\nabla(\phi-\phi_h)\|_{0,\omega_l} \Big).
 \end{align}

 Now we only need to estimate $\|R^h(p_h^i,\phi_h)\|_{0,\omega_l}$. Define
 \begin{align}\label{lower-jh-1-Rh}
  \widetilde{R}^h(p_h^i,\phi_h)|_\tau = {\rm div}\big(\gamma(x,p_h^i)\widetilde{G}_h\phi_h\big) + \frac{1}{|\tau|}\int_\tau {\rm div}\big(\alpha(x,p_h^i)\nabla p_h^i\big) +\frac{1}{|\tau|}\int_\tau{\rm div}\big(\beta(x,p_h^i)\big) - \frac{1}{|\tau|}\int_\tau g(x,p_h^i).
 \end{align}
 Since $\gamma(x,p_h^i)$ is assumed to be a linear function with respect to $p_h^i$, then $\widetilde{R}^h(p_h^i,\phi_h) \in \mathcal{P}^1(\tau)$ and
 \begin{align}\label{lower-jh-1-Rh-1}
  \|\widehat{R}^h(p_h^i,\phi_h) - \widetilde{R}^h(p_h^i,\phi_h)\|_{0,\tau}
  \le C h_\tau \big( \|\alpha(x,p_h^i)\|_{2,\infty,\tau}\|p_h^1\|_{1,\tau}
      + |\beta(x,p_h^i)|_{2,\tau} + |g(x,p_h^i)|_{1,\tau} \big).
 \end{align}
 On the other hand, from \eqref{a-ph-Jh-00}, for any $v\in H_0^1(\tau)$, we have
 \begin{align}\label{lower-jh-1-Rh-2}
  \big(\widetilde{R}^h(p_h^i,\phi_h),v\big)_\tau
  &=(\widehat{R}^h(p_h^i,\phi_h),v)_\tau- \big(\widehat{R}^h(p_h^i,\phi_h) - \widetilde{R}^h(p_h^i,\phi_h), v\big)_\tau\notag\\
  &=-\Big((a(p_h^i,v)_\tau-a(p^i,v)_\tau\Big)-\Big(b(p_h^i,\phi_h,v)_\tau-b(p^i,\phi,v)_\tau\Big)
   \notag\\
  &\quad - \big(\widehat{R}^h(p_h^i,\phi_h) - \widetilde{R}^h(p_h^i,\phi_h), v\big)_\tau -\Big( \big(\gamma(x,p_h^i)-\gamma(x,p^i)\big)
     (\widetilde{G}_h\phi_h-\nabla\phi_h),\nabla v \Big)_\tau \notag\\
    & \quad - \big(\gamma(x,p^i)(\widetilde{G}_h\phi_h-\nabla\phi_h), \nabla v\big)_\tau,
 \end{align}
 where we have used $\big(J_{h,l}(p_h^i),v\big)_l=0$, for any $l\in \partial \tau$.
 Then from \eqref{bound} and the similar arguments as in \eqref{a-ph-Jh-1}, we get
 \begin{align}\label{lower-jh-1-Rh-3}
 \big(\widetilde{R}^h(p_h^i,\phi_h),v\big)_\tau
 &\le C \big( \|p^i-p_h^i\|_{1,\tau}\|v\|_{1,\tau}
       + \|\widehat{R}^h(p_h^i,\phi_h) - \widetilde{R}^h(p_h^i,\phi_h)\|_{0,\tau}
         \|v\|_{0,\tau} \notag\\
 &\quad + \|\widetilde{G}_h\phi_h-\nabla\phi_h\|_{0,\tau}\|v\|_{1,\tau}
        + \|\nabla(\phi-\phi_h)\|_{0,\tau}\|v\|_{1,\tau} \big).
 \end{align}
 Taking $v=\mu_\tau\widetilde{R}^h(p_h^i,\phi_h)$ in \eqref{lower-jh-1-Rh-3} and by Lemma \ref{lem-lower}, there holds
 \begin{align*}
 \|\widetilde{R}^h(p_h^i,\phi_h)\|_{0,\tau}^2
 &\le C\big(\widetilde{R}^h(p_h^i,\phi_h),
             \mu_\tau\widetilde{R}^h(p_h^i,\phi_h) \big) \notag\\
 &\le C\big(  h_\tau^{-1} \|p^i-p_h^i\|_{1,\tau}
      + \|\widehat{R}^h(p_h^i,\phi_h) - \widetilde{R}^h(p_h^i,\phi_h)\|_{0,\tau}
         \notag\\
 &\quad + h_\tau^{-1} \|\widetilde{G}_h\phi_h-\nabla\phi_h\|_{0,\tau}
        + h_\tau^{-1}\|\nabla(\phi-\phi_h)\|_{0,\tau} \big) \|\widetilde{R}^h(p_h^i,\phi_h)\|_{0,\tau}.
 \end{align*}
 Hence,
 \begin{align}\label{lower-jh-1-Rh-5}
 \|\widetilde{R}^h(p_h^i,\phi_h)\|_{0,\tau}
 &\le C\big(  h_\tau^{-1} \|p^i-p_h^i\|_{1,\tau}
        + \|\widehat{R}^h(p_h^i,\phi_h) - \widetilde{R}^h(p_h^i,\phi_h)\|_{0,\tau} \notag\\
 &\quad + h_\tau^{-1}\|\widetilde{G}_h\phi_h-\nabla\phi_h\|_{0,\tau}
        + h_\tau^{-1}\|\nabla(\phi-\phi_h)\|_{0,\tau} \big).
 \end{align}
 Combining \eqref{lower-jh-1-Rh-1} and \eqref{lower-jh-1-Rh-5}, we have
 \begin{align}\label{lower-jh-1-Rh-6}
  \|\widehat{R}^h(p_h^i,\phi_h)\|_{0,\tau}
  &\le \|\widetilde{R}^h(p_h^i,\phi_h)\|_{0,\tau}
     + \|\widehat{R}^h(p_h^i,\phi_h) - \widetilde{R}^h(p_h^i,\phi_h)\|_{0,\tau} \notag\\
  &\le C \big(  h_\tau^{-1} \|p^i-p_h^i\|_{1,\tau} +  h_\tau^{-1}
          \|\widetilde{G}_h\phi_h-\nabla\phi_h\|_{0,\tau}
        +   h_\tau^{-1}\|\nabla(\phi-\phi_h)\|_{0,\tau} \big) \notag\\
  &\quad + C h_\tau \big( \|\alpha(x,p_h^i)\|_{2,\infty,\tau}\|p_h^1\|_{1,\tau}
        + |\beta(x,p_h^i)|_{2,\tau} + |g(x,p_h^i)|_{1,\tau} \big).
 \end{align}
 Inserting \eqref{lower-jh-1-Rh-6} into \eqref{lower-jh-1}, and using \eqref{jh-new-1}, we get
 \begin{align} \label{lower-jh-1new}
  h_{\omega_l}^{\frac{1}{2}}\|J_{h,l}(p_h^i)\|_{0,l}
  &\le C \big( \|p^i-p_h^i\|_{1,\omega_l}
    + \|\widetilde{G}_h\phi_h-\nabla\phi_h\|_{0,\omega_l}
    + \|\nabla(\phi-\phi_h)\|_{0,\omega_l} \big) \notag\\
  &\quad + C h_{\omega_l}^2\ \big( |\beta(x,p_h^i)|_{2,\omega_l} + |g(x,p_h^i)|_{1,\omega_l}+\|p_h^i\|_{1,\omega_l}\|\alpha(x,p_h^i)\|_{2,\infty,\omega_l} \big) \notag\\
  &\le C \Big( \|\nabla(p^i-p_h^i)\|_{0,\omega_\tau}
    + \|\nabla(\phi-\phi_h)\|_{0,\omega_\tau}
    + \sum\limits_{i=1}^n\|p^i-p_h^i\|_{0,\omega_\tau} \Big)
    + r_{p_h^i,\omega_\tau},
 \end{align}
 where $r_{p_h^i,\omega_l} \le C h_{\omega_\tau}^2 \big( \|p_h^i\|_{1,\omega_\tau}\|\alpha(x,p_h^i)\|_{2,\infty,\omega_\tau}+ |\beta(x,p_h^i)|_{2,\omega_\tau} + |g(x,p_h^i)|_{1,\omega_\tau} + |f|_{1,\omega_\tau}+\|\epsilon\|_{2,\infty,\omega_\tau}\|\phi_h\|_{1,\omega_\tau} \big)$.

 This completes the proof of Lemma \ref{lem-low-jhp0}.  $\hfill\Box$
\end{proof}

 By using the above lemma, now we present the lower bound for $\|\nabla(p^i-p_h^i)\|_{0,\omega_\tau}$.

\begin{theorem}\label{th-ph-lower}
 Let $(p^i, \phi)$ and $(p_h^i, \phi_h)$ be the solutions of \eqref{weak-p}-\eqref{weak-phi} and \eqref{fem-ph}-\eqref{fem-phih}, respectively. For any $l\in \partial \mathcal{T}^h$, $l\not\subset\partial\Omega$, if $h_l\ll 1$, the coefficients $\alpha,~\beta,~g$ and $\gamma$ of (\ref{MPNP-model}) satisfy
   \begin{align*}
  \|\alpha(x,p^i)-\alpha(x,p^i_h)\|_{0,\omega_l}&+\|\beta(x,p^i)-\beta(x,p^i_h)\|_{0,\omega_l}
  +\|g(x,p^i)-g(x,p^i_h)\|_{0,\omega_l}+\|\gamma(x,p^i)-\gamma(x,p^i_h)\|_{0,\omega_l}\notag\\
  &\leq C \|p^i-p_h^i\|_{0,\omega_l}
 \end{align*}
and $f\in H^1(\ome)$, then for any $ \tau\in \mathcal{T}^h$, there holds
 \begin{align} \label{th-ph-lower-0}	
  \eta_{\tau,p^i}(p_h^i,\phi_h)
  &\le C \widetilde{\xi}_{h,\omega_\tau}(p_h^i,\phi_h)
     \Big( \|\nabla(p^i-p_h^i)\|_{0,\omega_\tau}
    + \|\nabla(\phi-\phi_h)\|_{0,\omega_\tau}
    + \sum\limits_{i=1}^n\|p^i-p_h^i\|_{0,\omega_\tau} \Big) \notag\\
  &\quad  + \widetilde{r}_{p_h^i,\omega_\tau},
 \end{align}
 where
 \begin{align*}
 &\eta_{\tau,p^i}(p_h^i,\phi_h)
   = \|D_{h}(\phi_{h})\|_{0,\tau} + \|D_{h}(p_h^i)\|_{0,\tau}
    + h_{\tau}\big(\|R_{1h}(p^i_{h},\phi_{h})\|_{0,\tau}
    + \|R_{2h}(p^i_{h},\phi_{h})\|_{0,\tau} \big) \\
 &\quad\quad\quad\quad\quad\quad\quad +\|\gamma(x,p_h^i)(\widetilde{G}_h\phi_h-\nabla\phi_h)\|_{0,\tau}, \\
 & D_h(\phi_h)= G_h\phi_h - \epsilon(x)\nabla\phi_h, ~~~
   D_h(p_h^i) = G_hp_h^i-\alpha(x,p_h^i)\nabla p_h^i, \\
 & R_{1h}(p_h^i,\phi_h)=\sum_{i=1}^nq^ip_h^i+{\rm div}(G_h\phi_h)+f, \\
 & R_{2h}(p^{i}_{h},\phi_{h})= {\rm div}(G_hp_h^i)
       +{\rm div}\big(\beta(x,p_h^i)\big)-g(x,p_h^i)
       +{\rm div}\big(\gamma(x,p_h^i)\widetilde{G}_h\phi_h\big), \\
 & \widetilde{r}_{p_h^i,\omega_\tau}
    \le C \widetilde{\xi}_{h, \omega_{\tau}} h_{\omega_{\tau}}^{2}
     \Big( |p_h^i|_{1,\omega_l}\|\alpha(x,p_h^i)\|_{2,\infty,\omega_l}
     + |\beta(\cdot,p_h^i)|_{2,\omega_\tau} + |g(\cdot,p_h^i)|_{1,\omega_\tau} +\|\epsilon\|_{2,\infty,\omega_\tau}\|\phi_h\|_{1,\omega_\tau}+ |f|_{1,\omega_\tau}
     \Big), \\
 &\widetilde{\xi}_{h,\omega_\tau}(p_h^i,\phi_h)
    =\max\{ \widetilde{\xi}_{h, \omega_{\tau}}, \xi_{h, \omega_{\tau}} \}, \\
 & \widetilde{\xi}_{h, \omega_{\tau}}=\max _{l \in \partial \mathcal{T}^{h}, l
    \subset\omega_{\tau} \backslash \partial \omega_{\tau}} \widetilde{\xi}_{h, l}, ~~ \widetilde{\xi}_{h, l}=1+h_{l}|\alpha(\cdot,p_h^i)|_{1, \infty, l}
    \left\|\alpha^{-1}(\cdot,p_h^i)\right\|_{0, \infty, l}, \\
 & \xi_{h, \omega_{\tau}}=\max _{l \in \partial \mathcal{T}^{h}, l
    \subset \omega_{\tau} \backslash \partial \omega_{\tau}} \xi_{h, l}, ~~ \xi_{h,l}
    =1+h_{l}|\epsilon(x)|_{1, \infty, l}\|\epsilon(x)^{-1}\|_{0, \infty, l}.
 \end{align*}
\end{theorem}

\begin{proof}
 Similar to the proof of Theorem \ref{th-phi-lower}, first from the definition of $G_h$ (see \eqref{Gh-operator-alpha}), we get
 \begin{align*}
  G_hp_h^i-\alpha(x,p_h^i)\nabla p_h^i= \sum_{z\in\partial^2\mathcal{T}^h}\big(\alpha(x,p_h^i)\nabla p_h^i\big)_z\varphi_z-\alpha(x,p_h^i)\nabla p_h^i,
 \end{align*}
 where $\varphi_z$ is the basis function. Thus for any $x\in\tau, \tau\in\mathcal{T}^h$, suppose $\{z_i:i=1,2,3\}$ is the vertex set of $\tau$. Then by the similar arguments as \eqref{lower-Ghphi0}-\eqref{lower-phi03} in the proof of Theorem \ref{th-phi-lower}, we can easily get
 \begin{align} \label{lower-ph-Dhph}
  G_hp_h^i-\alpha(x,p_h^i)\nabla p_h^i
  &= \sum_{i=1}^{3} \varphi_{z_i}(x)\left(\sum_{j=1}^{J_{z_i}}
     \alpha_{z_{i}}^{j}\big(\alpha(z_i,p_h^i(z_i))\nabla p_h^i\big)_{\tau_{z_i}^{j}}\right) - \big(\alpha(x,p_h^i)\nabla p_h^i \big)(x) \notag\\
  &= \sum_{i=1}^{3} \varphi_{z_{i}}(x) \sum_{j=1}^{J_{z_{i}}}
    \alpha_{z_{i}}^{j}\left(\left(\alpha\left(z_{i}, p_h^i\left(z_{i}\right)\right) \nabla p_h^i\right)_{\tau_{z_i}^{j}}-\left(\alpha\left(x, p_h^i\right) \nabla p_h^i\right)_{\tau}\right) \notag\\
  &= \sum_{i=1}^{3} \varphi_{z_{i}}(x) \sum_{j=1}^{J_{z_{i}}} \alpha_{z_{i}}^{j}\left(\left(\alpha\left(z_{i}, p_h^i(z_i)\right) \nabla p_h^i\right)_{\tau_{z_i}^{j}} - \left(\alpha\left(z_{i}, p_h^i(z_i)\right) \nabla p_h^i\right)_{\tau}\right) \notag\\
  &\quad +\left(\sum_{i=1}^{3} \varphi_{z_{i}}(x)\left(\alpha\left(z_{i}, p_h^i(z_i)\right)\right)_{\tau} - \alpha(x, p_h^i)\right)\left(\nabla p_h^i \right)_{\tau}.
 \end{align}

 It is easy to know we can find a cluster of simplices $\tau', \tau_{1}, \cdots, \tau_{K}, \tau \in \omega_{z}$, such that $\overline{\tau}_{k} \cap \overline{\tau}_{k+1}= l_{k} \in \partial \mathcal{T}^{h}~(k=0,1, \cdots, K+1)$, where $\tau_0=\tau'$ and $\tau_{K+1}=\tau$. Thus  \begin{align*}
 &\big(\alpha(z_i,p_h^i(z_i)) \nabla p_h^i\big)_{\tau^{\prime}}
     -\big(\alpha(z_i,p_h^i(z_i))\nabla p_h^i\big)_{\tau} \\
 &=\sum_{k=0}^{K}\left(\big(\alpha(z_i,p_h^i(z_i)) \nabla p_h^i\big)_{\tau_{k}}
   -\big(\alpha(z_i,p_h^i(z_i)) \nabla p_h^i\big)_{\tau_{k+1}}\right) \\
 &=\sum_{k=0}^{K}\left(\big(\alpha(z_i,p_h^i(z_i) \nabla p_h^i\big)_{\tau_{k}}\cdot
   n_{l_{k}}-\big(\alpha(z_i,p_h^i(z_i)) \nabla p_h^i\big)_{\tau_{k+1}} \cdot n_{l_{k}}\right) n_{l_k}
\end{align*}
 or
 \begin{equation} \label{lower-ph-Dhph-1}
  \big(\alpha(z_i,p_h^i(z_i)) \nabla p_h^i\big)_{\tau^{\prime}}-\big(\alpha(z_i,p_h^i(z_i)) \nabla p_h^i\big)_{\tau}=\sum_{k=0}^{K}\left[\alpha(z_i,p_h^i(z_i)) \nabla p_h^i \cdot n_{l_{k}}\right] n_{l_{k}},
 \end{equation}
 where $n_{l_k}$ is the unit normal vector to $l_k$. If $z$ is a vertex of $l_k$, then
 \begin{align} \label{lower-ph-Dhph-2}
 & \Big\|\left[\alpha(z,p_h^i(z)) \nabla p_h^i \cdot
         n_{l_{k}}\right]\Big\|_{0, l_{k}} \notag\\
 &\le \Big\|[\alpha(\cdot,p_h^i) \nabla p_h^i \cdot n_{l_{k}}]\Big\|_{0,l_{k}} +\Big\|\left[\big(\alpha(\cdot,p_h^i)-\alpha(z,p_h^i(z))\big) \alpha^{-1}(\cdot,p_h^i) \alpha(\cdot,p_h^i) \nabla p_h^i \cdot n_{l_{k}}\right]\Big\|_{0, l_{k}}  \notag\\
 &\le C \widetilde{\xi}_{h, l_k}
      \Big\|\left[\alpha(\cdot,p_h^i) \nabla p_h^i \cdot n_{l_{k}}\right]\Big\|_{0, l_{k}},
 \end{align}
 where
 $\widetilde{\xi}_{h, l_k}=1+h_{l_k}|\alpha(\cdot,p_h^i)|_{1, \infty, l_k}\left\|\alpha^{-1}(\cdot,p_h^i)\right\|_{0, \infty, l_k}$.

 Obviously, we know that $\sum\limits_{i=1}^{3} \varphi_{z_{i}}(x)\big(\alpha(z_i,p_h^i(z_i))\big)_{\tau}$ is the Lagrange interpolation of $\alpha(x,p_h^i)$ and
 \begin{align} \label{lower-ph-Dhph-3}
  \left\|\left(\sum_{i=1}^{3}
   \varphi_{z_{i}}(x)\big(\alpha(z_i,p_h^i(z_i))\big)_{\tau}
   -\alpha(\cdot,p_h^i)\right) \nabla p_h^i \right\|_{0, \tau}
  \le C h_{\tau}^{2} |\alpha(\cdot,p_h^i)|_{2,\infty, \tau} |p_h^i|_{1,\tau}.
 \end{align}
 Then from \eqref{lem-low-jhp0-0} and \eqref{lower-ph-Dhph}-\eqref{lower-ph-Dhph-3}, we get
 \begin{align} \label{lower-ph-Dhph-4}
  \left\|D_h(p_h^i)\right\|_{0, \tau}
  &\le C  \left(\sum_{i=1}^{3} \sum_{l \in \partial \mathcal{T}^{h}, l
       \subset\omega_{z_{i}}} \xi_{h,l}h_l^{\frac 1 2}\|[\alpha(\cdot,p_h^i)\nabla p_h\cdot n_l]\|_{0,l}+ h_{\tau}^{2} |\alpha(\cdot,p_h^i)|_{2,\infty,\tau}|p_h^i|_{1,\tau}
       \right)\notag\\
 &\le C \widetilde{\xi}_{h, \omega_{\tau}}
        \Big( \|\nabla(p^i-p_h^i)\|_{0,\omega_\tau}
        + h_{\omega_\tau}\sum\limits_{i=1}^n\|p^i-p_h^i\|_{0,\omega_\tau}
        + \|\nabla(\phi-\phi_h)\|_{0,\omega_\tau} + r_{p_h^i,\omega_\tau} \Big)\notag\\
 &\quad + C h_{\tau}^{2} |\alpha(\cdot,p_h^i)|_{2,\infty,\tau}|p_h^i|_{1,\tau}
      \notag\\
 &\le C \widetilde{\xi}_{h, \omega_{\tau}}
     \Big( \|\nabla(p^i-p_h^i)\|_{0,\omega_\tau}
     + h_{\omega_\tau}\sum\limits_{i=1}^n\|p^i-p_h^i\|_{0,\omega_\tau}
     + \|\nabla(\phi-\phi_h)\|_{0,\omega_\tau} \Big)
     + \widetilde{r}_{p_h^i,\omega_\tau},
 \end{align}
 where
 \begin{align*}
  &\widetilde{r}_{p_h^i,\omega_\tau}
   \le C \widetilde{\xi}_{h, \omega_{\tau}} h_{\omega_{\tau}}^{2}
         \Big( \|\alpha(\cdot,p_h^i)\|_{2,\infty,\omega_\tau}\|p_h^i\|_{1,\omega_\tau}
        + |\beta(\cdot,p_h^i)|_{2,\omega_\tau} + |g(\cdot,p_h^i)|_{1,\omega_\tau}+\|\epsilon\|_{2,\infty,\omega_\tau}\|\phi_h\|_{1,\omega_\tau} +|f|_{1,\omega_\tau} \Big), \\
 &\widetilde{\xi}_{h, \omega_{\tau}}=\max _{l \in \partial \mathcal{T}^{h}, l \subset \omega_{\tau} \backslash \partial \omega_{\tau}} \widetilde{\xi}_{h, l}, ~~ \widetilde{\xi}_{h, l}=1+h_{l}|\alpha(\cdot,p_h^i)|_{1, \infty, l}\left\|\alpha^{-1}(\cdot,p_h^i)\right\|_{0, \infty, l}.
 \end{align*}

 Now we turn to estimate $\|R_{2h}(p^i_{h},\phi_{h})\|_{0,\tau}$. Define
 \begin{align*}
  \widetilde{R}_{2h}(p^i_{h},\phi_{h})|_\tau
  ={\rm div}(G_hp_h^i)|_\tau+{\rm div}\big(\gamma(x,p_h^i)\widetilde{G}_h\phi_h\big)|_\tau
       +\frac{1}{|\tau|}\int_\tau{\rm div}\big(\beta(x,p_h^i)\big)
       -\frac{1}{|\tau|}\int_\tau g(x,p_h^i).
 \end{align*}
 Since $\gamma(x,p_h^i)$ is assumed to be a linear function with respect to $p_h^i$, it is easy to see that $\widetilde{R}_{2h}(p^i_{h},\phi_{h}) \in \mathcal{P}^1(\tau)$ and
 \begin{align} \label{lower-R2htau}
  \|R_{2h}(p^i_{h},\phi_{h})-\widetilde{R}_{2h}(p^i_{h},\phi_{h})\|_{0,\tau}
  \le C h_\tau\big(|\beta(x,p_h^i)|_{2,\tau}
         +|g(x,p_h^i)|_{1,\tau}\big).
 \end{align}
 For any $v\in H_0^1(\Omega)$, taking $\chi=0$ in \eqref{upper-p-3}, we have
 \begin{align} \label{aeqn}
 a'(p_h^i;p^i-p_h^i,v)
 &= \big( D_h(p_h^i), \nabla v \big)
    + \big( R_{2h}(p_h^i,\phi_h), v \big) \notag \\
 &\quad + \big( \gamma(x,p_h^i)(\widetilde{G}_h\phi_h-\nabla\phi_h),
                \nabla v \big) - R(p_h^i,\phi_h,p^i,\phi,v),
 \end{align}
 where $$ R_{2h}(p^{i}_{h},\phi_{h})= {\rm div}(G_hp_h^i)
       +{\rm div}\big(\beta(x,p_h^i)\big)-g(x,p_h^i)
       +{\rm div}\big(\gamma(x,p_h^i)\widetilde{G}_h\phi_h\big). $$
 It follows from \eqref{upper-p-2} and \eqref{aeqn} that
 \begin{align} \label{lowph-2}
  \big( \widetilde{R}_{2h}(p^i_{h},\phi_{h}),v \big)_\tau
  &= (R_{2h}(p^i_{h},\phi_{h}),v)_\tau-\big( R_{2h}(p^i_{h},\phi_{h})-\widetilde{R}_{2h}(p^i_{h},\phi_{h}),v \big)_\tau\notag\\
   & = a'(p_h^i;p^i-p_h^i,v)_\tau - \big( D_h(p_h^i),\nabla v \big)_\tau
      -\big( R_{2h}(p^i_{h},\phi_{h})-\widetilde{R}_{2h}(p^i_{h},\phi_{h}),v \big)_\tau  \notag \\
 &\quad -\big( \gamma(x,p_h^i)(\widetilde{G}_h\phi_h-\nabla\phi_h),\nabla v \big)_\tau + R(p_h^i,\phi_h,p^i,\phi,v)_\tau \notag\\
 &=-a(p_h^i,v)_\tau-b(p_h^i,\phi_h,v)_\tau-\big( D_h(p_h^i),\nabla v \big)_\tau\notag\\
 &\quad -\big( R_{2h}(p^i_{h},\phi_{h})-\widetilde{R}_{2h}(p^i_{h},\phi_{h}),v \big)_\tau-\big( \gamma(x,p_h^i)(\widetilde{G}_h\phi_h-\nabla\phi_h),\nabla v \big)_\tau.
 \end{align}
 Then by using the similar arguments as in \eqref{bound}, we have
 \begin{align} \label{lowph-3}
 \big( \widetilde{R}_{2h}(p^i_{h},\phi_{h}),v \big)_\tau
 &\le C \big(\|p^i-p_h^i\|_{1,\tau}\|v\|_{1,\tau}
        +\|D_h(p_h^i)\|_{0,\tau}\|v\|_{1,\tau}
        +\|R_{2h}(p^i_{h},\phi_{h})-\widetilde{R}_{2h}(p^i_{h},\phi_{h})\|_{0,\tau}
          \|v\|_{0,\tau} \notag\\
 &\quad + \|\widetilde{G}_h\phi_h-\nabla\phi_h\|_{0,\tau}
          \|v\|_{1,\tau} + \|\nabla(\phi-\phi_h)\|_{0,\tau}\|v\|_{1,\tau} \big).
 \end{align}
 Setting $v=\mu_\tau\widetilde{R}_{2h}(p^i_{h},\phi_{h})$ in \eqref{lowph-3} and by Lemma \ref{lem-lower}, it yields
\begin{align*} 
 \|\widetilde{R}_{2h}(p^i_{h},\phi_{h})\|^2_{0,\tau}
 &\le C \big(\widetilde{R}_{2h}(p^i_{h},\phi_{h}),
              \mu_\tau\widetilde{R}_{2h}(p^i_{h},\phi_{h}) \big)_\tau    \notag \\
 &\le C \big( h_\tau^{-1}\|p^i-p_h^i\|_{1,\tau}
        + h_\tau^{-1}\|D_h(p_h^i)\|_{0,\tau}  +\|R_{2h}(p^i_{h},\phi_{h})-\widetilde{R}_{2h}(p^i_{h},\phi_{h})\|_{0,\tau}
           \notag\\
 &\quad + h_\tau^{-1}\|\widetilde{G}_h\phi_h-\nabla\phi_h\|_{0,\tau}
        +  h_\tau^{-1}\|\nabla(\phi-\phi_h)\|_{0,\tau} \big)
            \|\widetilde{R}_{2h}(p^i_{h},\phi_{h})\|_{0,\tau} .
 \end{align*}
 Hence,
 \begin{align} \label{lowph-5}
 \|\widetilde{R}_{2h}(p^i_{h},\phi_{h})\|_{0,\tau}
 &\le C \big( h_\tau^{-1}\|p^i-p_h^i\|_{1,\tau}
        +  h_\tau^{-1}\|D_h(p_h^i)\|_{0,\tau}
        + \|R_{2h}(p^i_{h},\phi_{h})-\widetilde{R}_{2h}(p^i_{h},\phi_{h})\|_{0,\tau} \notag\\
 &\quad  +  h_\tau^{-1}\|\widetilde{G}_h\phi_h-\nabla\phi_h\|_{0,\tau}
         +  h_\tau^{-1}\|\nabla(\phi-\phi_h)\|_{0,\tau} \big).
 \end{align}
 Then by \eqref{lower-R2htau} and \eqref{lowph-5}, we obtain
 \begin{align*} 
  \|R_{2h}(p^i_{h},\phi_{h})\|_{0,\tau}
  &\le \|\widetilde{R}_{2h}(p^i_{h},\phi_{h})\|_{0,\tau}
      + \|R_{2h}(p^i_{h},\phi_{h})
      - \widetilde{R}_{2h}(p^i_{h},\phi_{h})\|_{0,\tau} \notag\\
  &\le C \big( h_\tau^{-1}\|p^i-p_h^i\|_{1,\tau}
        + h_\tau^{-1}\|D_h(p_h^i)\|_{0,\tau}
        + h_\tau\big(|\beta(x,p_h^i)|_{2,\tau}
        + |g(x,p_h^i)|_{1,\tau}\big) \notag\\
  &\quad + h_\tau^{-1}\|\widetilde{G}_h\phi_h-\nabla\phi_h\|_{0,\tau}
        + h_\tau^{-1}\|\nabla(\phi-\phi_h)\|_{0,\tau} \big)
 \end{align*}
 or
 \begin{align} \label{lowph-6-0}
 h_\tau\|R_{2h}(p^i_{h},\phi_{h})\|_{0,\tau}
 &\le C \big( \|p^i-p_h^i\|_{1,\tau}
        +  \|D_h(p_h^i)\|_{0,\tau}
        + h_\tau^2 \big( |\beta(x,p_h^i)|_{2,\tau} + |g(x,p_h^i)|_{1,\tau} \big) \notag\\
 &\quad + \|\widetilde{G}_h\phi_h-\nabla\phi_h\|_{0,\tau}
        + \|\nabla(\phi-\phi_h)\|_{0,\tau} \big).
 \end{align}
 Combining \eqref{lower-ph-Dhph-4} and \eqref{lowph-6-0}, and using \eqref{jh-new-1}, we get
 \begin{align} \label{lowph-7}
 h_\tau\|R_{2h}(p^i_{h},\phi_{h})\|_{0,\tau}
 &\le C \widetilde{\xi}_{h, \omega_{\tau}}
     \Big( \|\nabla(p^i-p_h^i)\|_{0,\omega_\tau}
    + \|\widetilde{G}_h\phi_h-\nabla\phi_h\|_{0,\tau}
    + \|\nabla(\phi-\phi_h)\|_{0,\omega_\tau} \notag\\
 &\quad   + \sum\limits_{i=1}^n\|p^i-p_h^i\|_{0,\omega_\tau} \Big)
    + \widetilde{r}_{p_h^i,\omega_\tau} \notag\\
 &\le C \widetilde{\xi}_{h, \omega_{\tau}}
     \Big( \|\nabla(p^i-p_h^i)\|_{0,\omega_\tau}
    + \|\nabla(\phi-\phi_h)\|_{0,\omega_\tau}
    + \sum\limits_{i=1}^n\|p^i-p_h^i\|_{0,\omega_\tau} \Big)
         + \widetilde{r}_{p_h^i,\omega_\tau}.
 \end{align}

 Therefore, the desired estimate \eqref{th-ph-lower-0} is obtained from \eqref{th-phi-lower-0}, \eqref{lower-ph-Dhph-4} and \eqref{lowph-7}. $\hfill\Box$
\end{proof}

 Similar to Remark \ref{Remark-upper}, we have the following results.
\begin{remark} \label{Remark-lower}
 If $\|p^i-p_h^i\|_{0,\omega_\tau}\le Ch_{\omega_\tau}^2$, then from Theorems
  \ref{th-phi-lower} and \ref{th-ph-lower}, it yields
 \begin{align*}
  \eta_{\tau,\phi}(p^i_{h},\phi_{h})
  \le C\big(\|\nabla(\phi-\phi_h)\|_{0,\omega_{\tau}} + h_{\omega_\tau}^2 \big),
 \end{align*}
 and
 \begin{align*}
  \eta_{\tau,p^i}(p^i_{h},\phi_{h})
  \le C\big(\|\nabla(p^i-p^i_{h})\|_{0,\omega_{\tau}}
      + \|\nabla(\phi-\phi_h)\|_{0,\omega_{\tau}}
      +  h_{\omega_\tau}^2 \big).
 \end{align*}
\end{remark}

\subsection{Adaptive algorithm} \label{sec-adaptive}

\noindent
 In this subsection, we describe a typical adaptive finite element algorithm based on the a posteriori error estimators derived above.

 For $\tau\in \mathcal{T}^h$, we denote the local error indicators for the electrostatic potential and concentrations respectively by
 \begin{align} \label{adp-indi-phih}
    \eta_{\tau,\phi}(p^i_{h},\phi_h)
  &= \|D_h(\phi_h)\|_{0,\tau}
      +h_\tau\|R_{1h}(p^i_{h},\phi_h)\|_{0,\tau}, \\
       \notag \\
    \eta_{\tau,p^i}(p_h^i,\phi_h)
  &= \|D_h(p_h^i)\|_{0,\tau} + \|D_h(\phi_h)\|_{0,\tau}
       + \|\gamma(x,p_h^i)(\widetilde{G}_h\phi_h-\nabla\phi_h)\|_{0,\tau} \notag \\
  &\quad +h_\tau\big(\|R_{1h}(p^i_{h},\phi_{h})\|_{0,\tau}
         +\|R_{2h}(p^i_{h},\phi_{h})\|_{0,\tau}\big),
         \label{adp-indi-ph-up}
 \end{align}
 where
 \begin{align*}
 &D_{h}(\phi_{h})=  G_{h}\phi_{h} - \epsilon(x)\nabla\phi_{h}, ~~~
  D_h(p_h^i)= G_hp_h^i - \alpha(x,p_h^i)\nabla p_h^i,  \\
 &R_{1h}(p^{i}_{h},\phi_{h})=\sum_{i=1}^nq^ip^i_{h}
       +{\rm div}(G_{h}\phi_h)+f, \\
 &R_{2h}(p^{i}_{h},\phi_{h})= {\rm div}(G_hp_h^i)
       +{\rm div}\big(\beta(x,p_h^i)\big)-g(x,p_h^i)
       +{\rm div}\big(\gamma(x,p_h^i)\widetilde{G}_h\phi_h\big).
 \end{align*}

 Given an initial conforming mesh $\mathcal{T}^h$, an associated finite element space $S_0^h$ and a tolerance $TOL$, the typical adaptive finite element algorithm is then designed as follows:
\begin{algorithm}[H]
 \caption{Adaptive Computing for nonlinear PNP equations }
  \label{adaptive-algorithm}
   \begin{itemize}
     \item Step 1: Computing the finite element solution  \\
         Find the finite element solution $p_h^i,~ i=1,2\cdots,n $ and $\phi_h\in S_0^h$.
     \item Step 2: Error estimation \\
         Compute the local error indicators $\eta_{\tau,\phi}$ and $\eta_{\tau,p^i}$ by \eqref{adp-indi-phih} and \eqref{adp-indi-ph-up} respectively for all $\tau\in \mathcal{T}^h$.
     \item Step 3: Local refinement  \\
         If $\Big(\sum\limits_{\tau\in\mathcal{T}^h}
         \eta^2_{\tau,\phi}\Big)^{\frac{1}{2}} > TOL$ or $\Big(\sum\limits_{\tau\in\mathcal{T}^h}
         \eta^{2}_{\tau,p^i}\Big)^{\frac{1}{2}} > TOL$, then refine those elements which satisfy
         $\eta_{\tau,\phi}\ge \theta\max\limits_{\tau\in \mathcal{T}^h}\eta_{\tau,\phi}$ or $\eta_{\tau,p^i}\ge \theta\max\limits_{\tau\in \mathcal{T}^h}\eta_{\tau,p^i}$  with $\theta\in(0,1)$ is a given refinement parameter.
     \item Step 4: Generating a new mesh \\
          Generate a new mesh $\mathcal{T}^h$, a space $S_0^h$ and return to Step 1. Otherwise, the computation is terminated.
  \end{itemize}
\end{algorithm}
In our computations, we follow the refining strategies in \cite{R.Ver1994,R.Ver1996} for two dimensions to obtain a new conforming mesh and choose the refinement parameter $\theta = 0.5$.

\setcounter{equation}{0}
\section{Numerical experiments} \label{sec-Numer-experiments}

\noindent
 In this section, we will report the numerical results to illustrate the theoretical results obtained in Section \ref{sec-posteriori}. First, the true errors of the finite element solutions and the error estimators are compared both on the uniform meshes and the adaptive meshes for a nonlinear PNP model with a smooth solution. Then an example with a singular solution is reported to show the efficiency of the adaptive computation proposed in this paper.

 Denote by $\eta_{\phi}$, $ \eta_{p^{1}}$ and $ \eta_{p^{2}}$ the a posteriori error estimators for the electrostatic potential $ \phi$, the positive ion concentration $p^1$ and the negative ion concentration $p^2$, respectively.  Let $\mathcal{T}^h=\{\tau\}$ be a shape-regular mesh of $ \Omega$ with mesh size $ h>0$ and $ \tau$ be the element.
 Define
 \begin{align*}
  e_\phi&=\|\phi-\phi_h\|_{1,\Omega}, ~~~~~~~~~~ e_{p^i}=\|p^i-p_h^i\|_{1,\Omega}, \\
  \eta_\phi&=\left(\sum\limits_{\tau\in \mathcal{T}^h}\eta^2_{\tau,\phi}\right)^{1/2},~~~~~
  \eta_{p^i}=\left(\sum\limits_{\tau\in \mathcal{T}^h}\eta^2_{\tau,p^i}\right)^{1/2},
 \end{align*}
 where $\eta_{\tau,\phi}$ and $\eta_{\tau,p^i}$ are defined in \eqref{adp-indi-phih} and \eqref{adp-indi-ph-up}, respectively. In particular, in the following, we use symbols $e_{u,\phi}$ and $e_{u,p^i}$ represent the errors on uniform meshes, and $e_{a,\phi}, e_{a,p^i}$ represent the errors on adaptive meshes for the electrostatic potential and concentrations, respectively. Correspondingly, the symbols $\eta_{u,\phi}, \eta_{u,p^i}$ are used to denote the error estimators on uniform meshes, and $\eta_{a,\phi}, \eta_{a,p^i}$ denote the error estimators on adaptive meshes, respectively.


\begin{figure}[H]
  \centerline{
  \includegraphics[height=6.2cm,width=18.0cm]
    {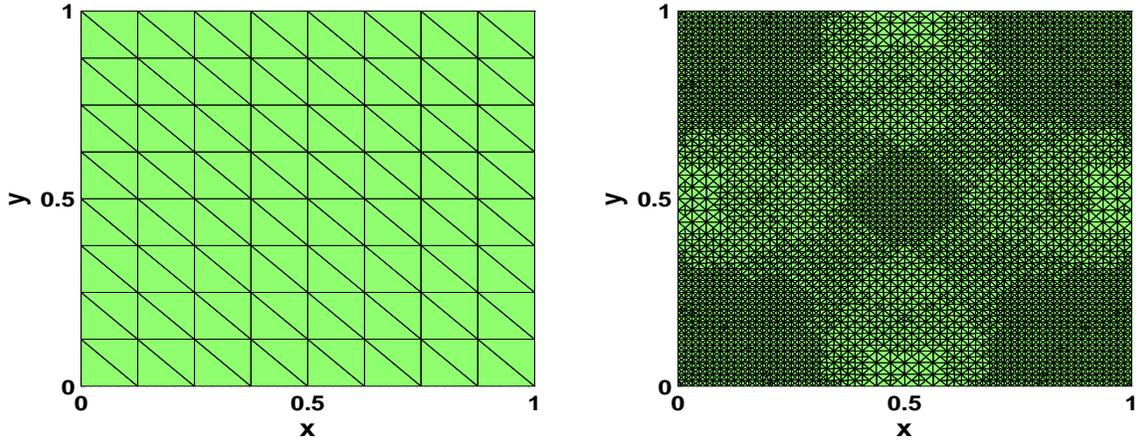}   }
  \caption{The left figure is the initial uniform mesh with 81 degrees of freedom and the right one is an adaptive mesh with 6,253 degrees of freedom constructed by the error indicators $\eta_{\tau,\phi}$ and $\eta_{\tau,p^i}$ for Example \ref{exam2d_MNPNP-exact}. }
  \label{sech2d-smooth_mesh_ini81_ada6253}
\end{figure}

\begin{example}\label{exam2d_MNPNP-exact}
 Consider the following steady-state nonlinear PNP equations, which is a simplified form of the PNP equations for ion channel (cf. \cite{Y.K.Hyon2014})
 \begin{align} \label{MNPNP-exact-sech2p}
  \left\{\begin{array}{l}
  -\nabla\cdot\big(\nabla p^i + p^i\nabla({\rm sech}^2p^i) +q^ip^i\nabla\phi \big) = f_i,
   ~~{\rm in}~ \Omega,~~i=1,2,  \\
  -\Delta\phi-\displaystyle{\sum_{i=1}^2}q^ip^i=f_3,
   ~~{\rm in}~ \Omega.
  \end{array}\right.
 \end{align}
 Here the computational domain $\Omega=[0,1]^2\subset \mathbb{R}^2$, $q^1=1$ and $q^2=-1$. The boundary condition and the right-hand side functions are chosen such that the exact solution $(\phi,p^1,p^2)$ is given by
 \begin{align} \label{mpnp-exact_sin}
  \left\{\begin{array}{ll}
  \phi=\sin(\pi x) \sin(\pi y), \\
  p^1=\sin(2\pi x) \sin(2\pi y), \\
  p^2=\sin(3\pi x) \sin(3\pi y).
 \end{array}\right.
 \end{align}
\end{example}
 First by a simple calculation, the first equation in \eqref{MNPNP-exact-sech2p} can be rewritten as
 {\setlength\abovedisplayskip{15pt}
\setlength\belowdisplayskip{15pt}  
 \bea \label{mpnp-changed-exam}
  \mathcal{L}(p^i,\phi)\equiv -\nabla\cdot\big((1-2p^i\tanh p^i {\rm sech}^2p^i)\nabla p^i
  +q^ip^i\nabla\phi \big) - f_i =0, ~~\mbox{in}~~ \Omega,~~i=1,2.
 \eea
 Then from \eqref{MPNP-model} and \eqref{mpnp-changed-exam}, we see that $\alpha(x, p^i) = 1-2p^i\tanh p^i {\rm sech}^2p^i, ~ \beta(x,p^i) = 0, ~\gamma(x,p^i)= q^ip^i, ~g(x,p^i)= -f_i$, $\epsilon(x) = 1$, and $f=f_3$ in this example. In addition, by \eqref{mpnp-exact_sin}, we know that $p^i\in[0,1]$, so that $\tanh p^i{\rm sech}^2p^i < \frac{1}{2}$ and
 $\alpha(x,p^i)\ge 0$. Hence the assumption \eqref{elliptic-condition} is satisfied, which indicates that $\mathcal{L}'(p,\phi)$ is isomorphic. According to Lax-Milgram theorem, it follows that the solution $(p^i,\phi)$ is unique.

 \vspace{-3em} 

 This example is mainly used to verify the reliability of the error indicators. The initial uniform mesh and an adaptive mesh constructed by the error indicators $\eta_{\tau,\phi}$ and $\eta_{\tau,p^i}$ for Example \ref{exam2d_MNPNP-exact} are shown in Fig. \ref{sech2d-smooth_mesh_ini81_ada6253}. The numerical results on the uniform meshes and the adaptive meshes for the electrostatic potential $\phi$, the concentrations $p^1$ and $p^2$ are presented in Figs. \ref{sech-smooth-phi}, \ref{sech-smooth-p1} and \ref{sech-smooth-p2}, respectively. It is apparent from Fig. \ref{sech-smooth-phi} that the a posterior error estimators of the electrostatic potential $\phi$ approximate the true errors as the increase of the degrees of freedom both on the uniform meshes and the adaptive meshes. On the other hand, it is also shown that the error curves of the electrostatic potential keep the quasi-optimal convergence order (since the error curves are parallel to the quasi-optimal convergence curve with slope of $-\frac{1}{2}$), which verifies the theoretical results shown in Lemma \ref{the-phi-femH1} and Theorem \ref{theo-upper-phi}. Similarly, for the concentrations $p^1$ and $p^2$, we can get the similar results, see Figs. \ref{sech-smooth-p1} and \ref{sech-smooth-p2}.

\begin{figure}[H]
  \centerline{
  \includegraphics[height=7.5cm,width=18.0cm]
    {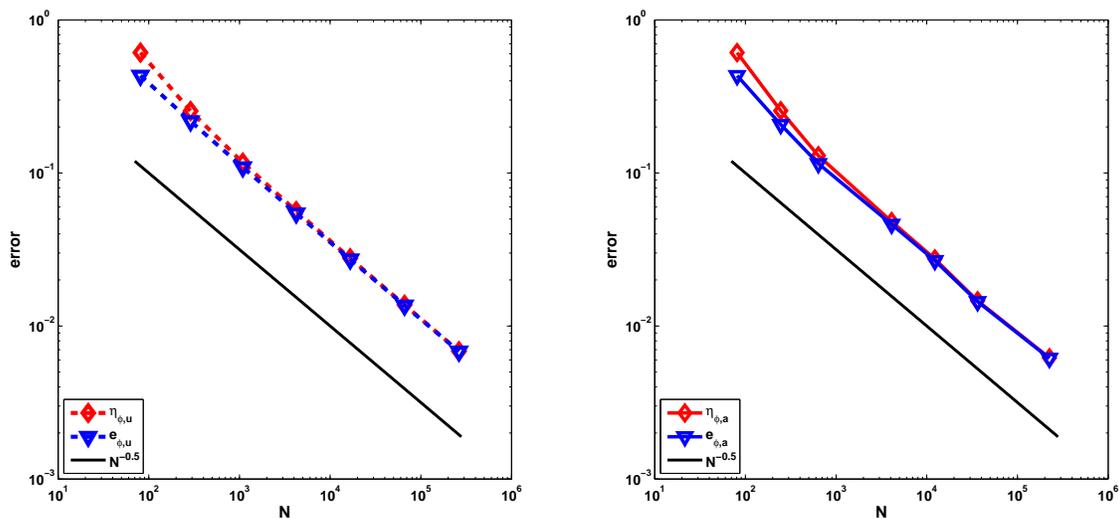}   }
  \caption{The error indicators and $H^1$ norm errors of the electrostatic potential $\phi$ on uniform meshes and adaptive meshes for Example \ref{exam2d_MNPNP-exact}. The black solid line is a quasi-optimal convergence curve with slope $-\frac{1}{2}$ and $N$ is the number of degrees of freedom. }
  \label{sech-smooth-phi}
\end{figure}

\begin{figure}[H]
  \centerline{
  \includegraphics[height=7.5cm,width=18.0cm]
    {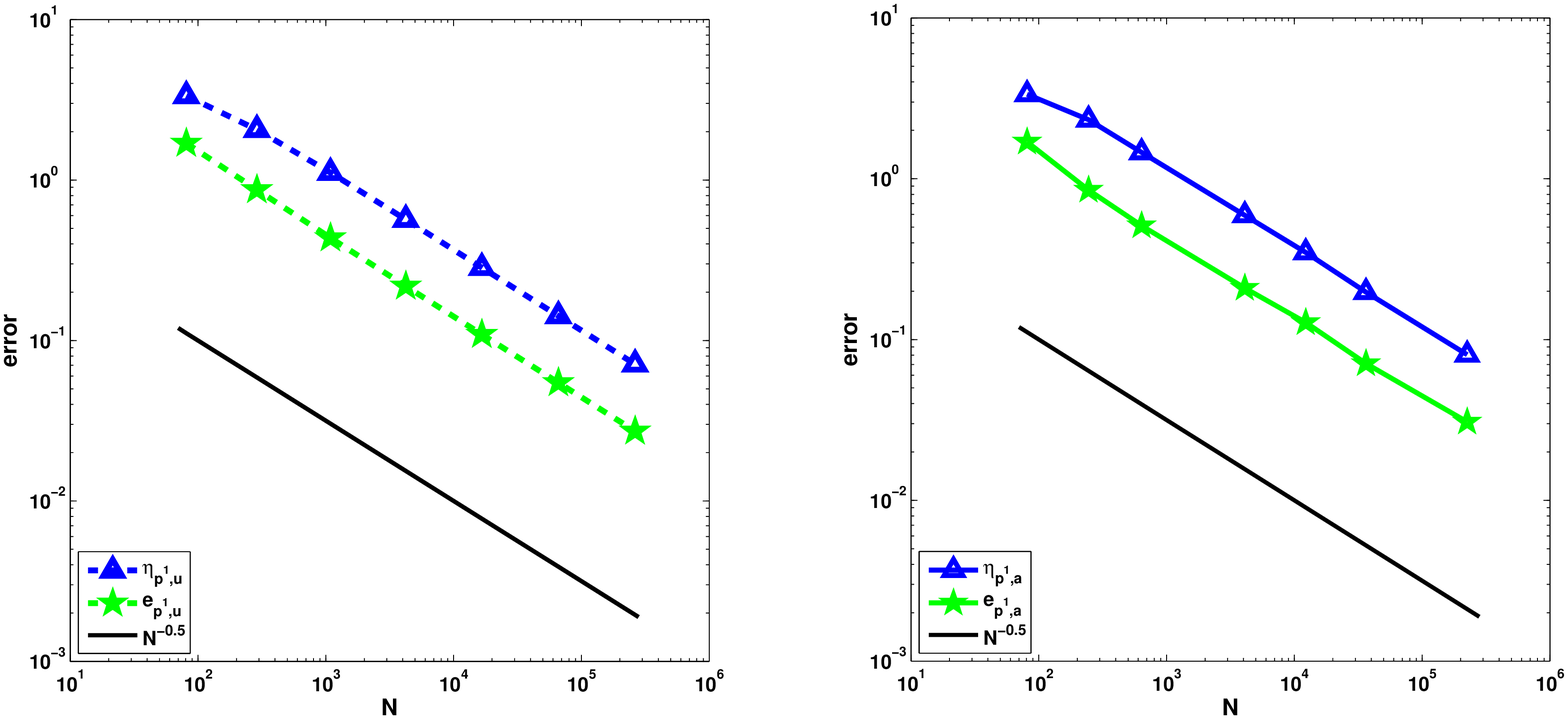}   }
   \caption{The error indicators and $H^1$ norm errors of the positive ion concentration $p^1$ on uniform meshes and adaptive meshes for Example \ref{exam2d_MNPNP-exact}.  The black solid line is a quasi-optimal convergence curve with slope $-\frac{1}{2}$ and $N$ is the number of degrees of freedom. }
   \label{sech-smooth-p1}
\end{figure}

\begin{figure}[H]
  \centerline{
  \includegraphics[height=7.5cm,width=18.0cm]
    {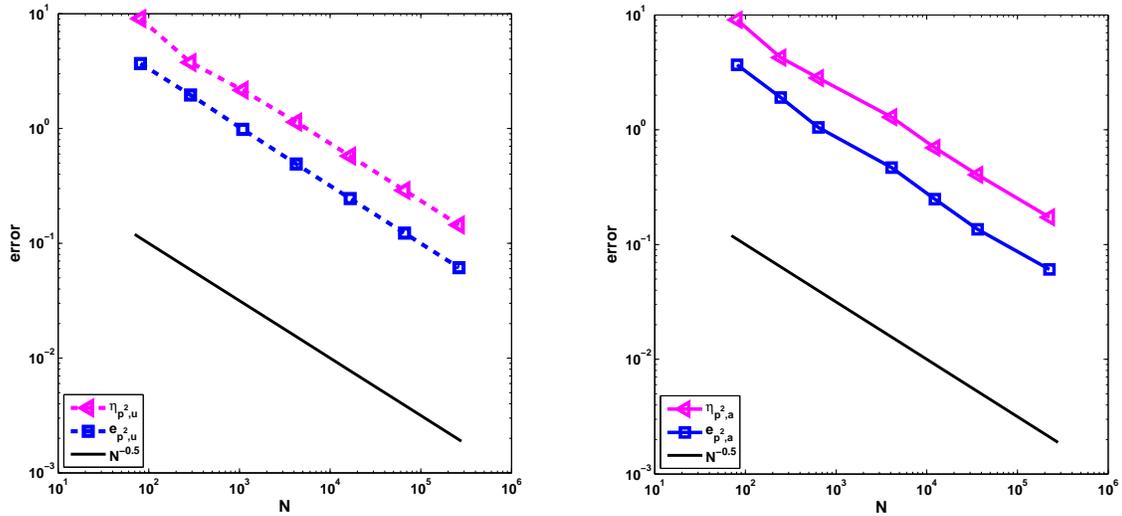}   }
   \caption{The error indicators and $H^1$ norm errors of the negative ion concentration $p^2$ on uniform meshes and adaptive meshes for Example \ref{exam2d_MNPNP-exact}. The black solid line is a quasi-optimal convergence curve with slope $-\frac{1}{2}$ and $N$ is the number of degrees of freedom. }
   \label{sech-smooth-p2}
\end{figure}

In the above, we have presented the example with a smooth solution to verify the reliability of the a posteriori error indicators. In the following, we consider another example of which the exact solution has a strong singularity at the point $(x_0,y_0) = (0,0)$.

\begin{example}\label{exam2dPNP-phisingular-p3}
 Consider the following nonlinear PNP equations with a singular point $(0,0)$:
 \begin{align} \label{2dPNP-singular-equation-p3}
  \left\{\begin{array}{l}
  -\nabla\cdot\left(\nabla p_i+q_ip_i\nabla\phi \right) + p_i^3=f_i,
   ~~{\rm in}~~ \Omega,~~i=1,2,  \\
  -\Delta\phi-\displaystyle{\sum_{i=1}^2}q_ip_i=f_3,
   ~~{\rm in}~~\Omega,
  \end{array}\right.
 \end{align}
 where $\Omega=[0,1]^2\subset \mathbb{R}^2$ and $q_1=1,~q_2=-1$. Compared \eqref{2dPNP-singular-equation-p3} with \eqref{MPNP-model}, it is seen that $\alpha(x, p^i) = 1, \beta(x,p^i) = 0, \gamma(x,p^i)= q^ip^i, g(x,p^i)= p_i^3-f_i$, $\epsilon(x) = 1$ and $f=f_3$. The boundary condition and the right-hand side functions are chosen such that the exact solution $(\phi,p_1,p_2)$ is given by
 \begin{align} \label{2dpnp-exact-p3}
  \left\{\begin{array}{ll}
  \phi~=(x^2+y^2)^{0.1}, \\
   p_1=\frac{\sin(2\pi x) \sin(2\pi y)} {2x^2 + 2y^2}, \\
   p_2=\frac{\sin(3\pi x) \sin(3\pi y)} {2x^2 + 2y^2}.
 \end{array}\right.
 \end{align}
\end{example}

\begin{figure}[H]
  \centerline{
  \includegraphics[height=4.5cm,width=18.0cm]
    {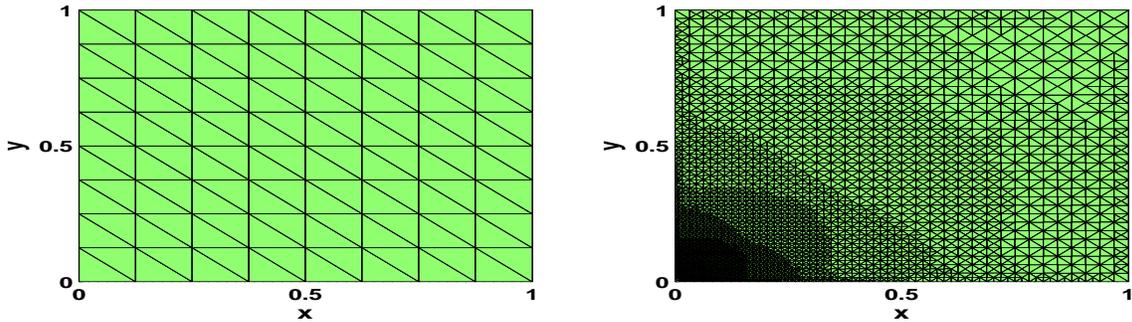}   }
  \caption{The left figure is the initial uniform mesh with 81 degrees of freedom and the right one is an adaptive mesh with 7,432 degrees of freedom for Example \ref{exam2dPNP-phisingular-p3}. }
  \label{p3-mesh2d-mpnp}
\end{figure}

 Fig. \ref{p3-mesh2d-mpnp} shows the initial uniform mesh with 81 degrees of freedom (left) and an adaptive mesh with 7,432 degrees of freedom (right). It is shown by Fig. \ref{p3-mesh2d-mpnp} that the adaptive mesh is locally refined near the origin which coincides with the position of the singularity at the point $(0,0)$.

 The numerical results of the electrostatic potential $\phi$ and the concentrations $p_1,~p_2$ on uniform meshes and adaptive meshes are presented in Figs. \ref{err2d-p3-uniada_phi}, \ref{err2d-p3-uniada_p1} and \ref{err2d-p3-uniada_p2}, respectively. It is observed from Figs. \ref{err2d-p3-uniada_phi}, \ref{err2d-p3-uniada_p1} and \ref{err2d-p3-uniada_p2} that the errors on adaptive meshes (solid line) are much less than that on uniform meshes (dashed line), which indicates the efficiency of Algorithm \ref{adaptive-algorithm}. For example, for the electrostatic potential $\phi$, it is shown in Fig. \ref{err2d-p3-uniada_phi} that the error value $e_{\phi} \le 0.086$ achieved with about 130 degrees of freedom on the adaptive mesh. However, it costs about 260,000 degrees of freedom on the uniform mesh to achieve the same accuracy. The ratio of degrees of freedom is about $1:2,000$. For the concentrations $p_1$ and  $p_2$, similar results can be obtained from Figs. \ref{err2d-p3-uniada_p1} and \ref{err2d-p3-uniada_p2}. On the other hand, it is shown from Figs. \ref{err2d-p3-uniada_phi}, \ref{err2d-p3-uniada_p1} and \ref{err2d-p3-uniada_p2} that the convergence orders of the error curves (solid line) for the true errors and the error estimators on adaptive meshes are quasi-optimal both for the electrostatic potential and concentrations, which indicates the adaptive finite element computation based on the a posteriori error indicators derived in this paper is efficient for the nonlinear PNP system with a singular solution.

\begin{figure}[H]
  \centerline{
  \includegraphics[height=8.0cm,width=10.0cm]
    {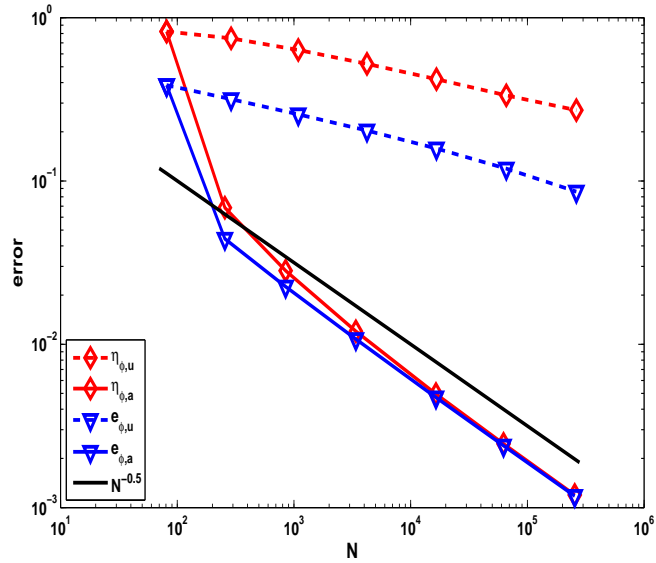}   }
  \caption{The $H^1$ norm errors and error indicators of the electrostatic potential $\phi$ versus the degrees of freedom $N$ of the mesh for Example \ref{exam2dPNP-phisingular-p3} by the uniform refinement (dashed line) and adaptive refinement (solid line). The black solid line is a quasi-optimal convergence curve with slope $-\frac{1}{2}$. }
  \label{err2d-p3-uniada_phi}
\end{figure}

\vspace{-0.8cm}
\begin{figure}[H]
  \centerline{
  \includegraphics[height=8.0cm,width=10.0cm]
    {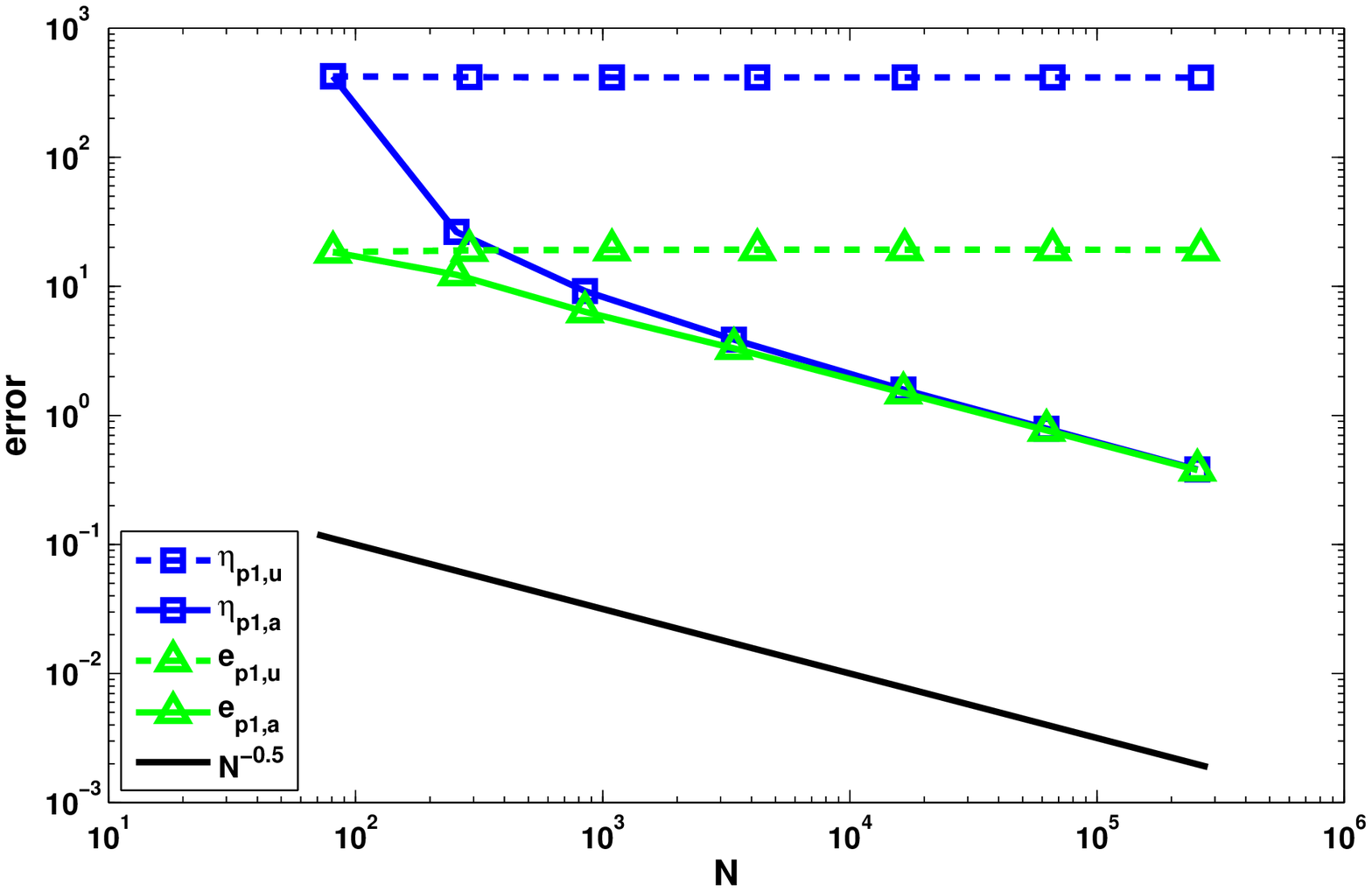}   }
  \caption{The $H^1$ norm errors and error indicators of the positive ion concentration $p_1$ versus the degrees of freedom $N$ of the mesh for Example \ref{exam2dPNP-phisingular-p3} by the uniform refinement (dashed line) and adaptive refinement (solid line). The black solid line is a quasi-optimal convergence curve with slope $-\frac{1}{2}$.}
  \label{err2d-p3-uniada_p1}
\end{figure}

\vspace{-0.8cm}
\begin{figure}[H]
  \centerline{
  \includegraphics[height=8.0cm,width=10.0cm]
    {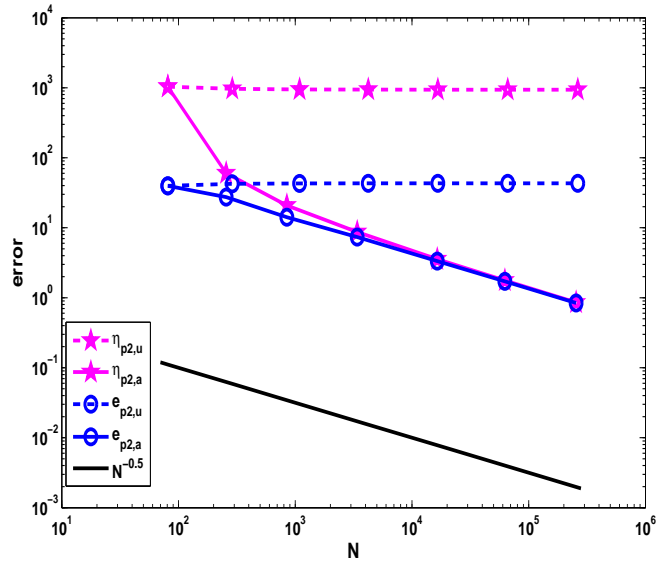}   }
  \caption{The $H^1$ norm errors and error indicators of the negative ion concentration $p_2$ versus the degrees of freedom $N$ of the mesh for Example \ref{exam2dPNP-phisingular-p3} by the uniform refinement (dashed line) and adaptive refinement (solid line). The black solid line is a quasi-optimal convergence curve with slope $-\frac{1}{2}$.}
  \label{err2d-p3-uniada_p2}
\end{figure}

\section{Conclusion} \label{sec-conclusion}

\noindent
 In this paper, we have derived a local averaging type a posteriori error estimators for a class of nonlinear steady-state Poisson-Nernst-Planck equations. Both the global upper bounds and the local lower bounds of the a posteriori error estimators are obtained for the electrostatic potential and concentrations. It is shown by the theoretical analysis and numerical experiments that the adaptive finite element computation based on the a posteriori error estimators is efficient and reliable. 
 The a posteriori error analysis and the corresponding adaptive finite element algorithms can be extended to more general and complex nonlinear PNP equations, for example, the coefficients $\alpha(\cdot,p^i)$ and $\epsilon(x)$ can be discontinuous coefficients or piecewise constants, which will be discussed in our next work for practical ion channel problems.


\begin{acknowledgements}
 Deep thanks must be expressed to Professor Benzhuo Lu for his helpful discussions and valuable suggestions. Thanks also go to Ming Tang for his help on numerical experiments. Y. Yang was supported by the China NSF (NSFC 11561016, NSFC 11771105), Guangxi Colleges and Universities Key Laboratory of Data Analysis and Computation open fund and the Hunan Key Laboratory for Computation and Simulation in Science and Engineering, Xiangtan University. S. Shu was supported by the China NSF (NSFC 11971414).  R. G. Shen was supported by Postgraduate Scientific Research and Innovation Fund of the Hunan Provincial Education Department (CX2017B268).
\end{acknowledgements}

%
%



\end{document}